\documentclass[a4paper,12pt]{amsart}

\usepackage{hyperref}

\usepackage{amsfonts}
\usepackage{amssymb}
\usepackage{amsxtra}
\usepackage{amstext}
\usepackage[english]{babel}
\usepackage{enumerate}

\usepackage{xcolor,ulem}

\newtheorem{theorem}{Theorem}[section]

\newtheorem{lemma}[theorem]{Lemma}
\newtheorem{corollary}[theorem]{Corollary}
\newtheorem{proposition}[theorem]{Proposition}
\newtheorem{example}[theorem]{Example}
\newtheorem{remark}[theorem]{Remark}

\newtheorem{hypothesis}[theorem]{Hypothesis}

\def\bit{\begin{itemize}}
\def\eit{\end{itemize}}
\reversemarginpar   %%puts label in LEFT margin
\def\bc{\begin{center}}
\def\ec{\end{center}}
\def\bthm{\begin{theorem}}
\def\ethm{\end{theorem}}
\def\bcor{\begin{corollary}}
\def\ecor{\end{corollary}}
\def\bprop{\begin{proposition}}
\def\eprop{\end{proposition}}
\def\blem{\begin{lemma}}
\def\elem{\end{lemma}}

\def\brem{\begin{remark}}
\def\erem{\end{remark}}

\def\bdes{\begin{description}}
\def\edes{\end{description}}

\def\beq{\begin{equation}}
\def\eeq{\end{equation}}
\def\ben{\begin{enumerate}}
\def\een{\end{enumerate}}
\def\beqar{\begin{eqnarray}}
\def\eeqar{\end{eqnarray}}
\def\beqarr{\begin{eqnarray*}}
\def\eeqarr{\end{eqnarray*}}

\def\P{{\mathsf P}} %Probability
\def\E{{\mathsf E}} % Esperance
\def\a{\alpha}

\def\rmd{{\rm d}}

\def\part{\partial}

\def\d#1dt{\frac{d#1}{dt}}    %%variable ODE left side

\begin{document}

\title{Strongly vertex-reinforced jump process on a complete graph}
\author{Olivier Raimond} 
\address{(O. Raimond) Mod\'elisation al\'eatoire de l'Universit\'e Paris Nanterre (MODAL'X), 92000 Nanterre, France}
\email{olivier.raimond@parisnanterre.fr}
\author{Tuan-Minh Nguyen}
\address{(T.M. Nguyen) School of Mathematics, Monash University, 3800 Victoria, Australia}
\email{tuanminh.nguyen@monash.edu}
\date{\today}

\begin{abstract}
The aim of our work is to study vertex-reinforced jump processes with super-linear weight function $w(t)=t^{\alpha}$, for some $\alpha>1.$ On any complete graph $G=(V,E)$, we prove that there is one vertex $v\in V$ such that the total time spent at $v$ almost surely tends to infinity while the total time spent at the remaining vertices is bounded.\\ 

\textbf{Résumé.} Le but de notre travail est d'étudier les processus de sauts renforcés par sites par une fonction de poids sur-linéaire $w(t)= t^{\alpha}$, avec $\alpha>1$. Sur tout graphe complet $G = (V, E)$, on montre qu'il y a un sommet $v \in V$ tel que le temps total passé en $v$ tend presque sûrement vers l'infini tandis que le temps total passé dans les sommets restants est borné.
\end{abstract}

\maketitle
%\tableofcontents

\section{Introduction}\label{sec:introduction}

Let $G=(V,E)$ be a finite connected, undirected graph without loops, where $V=\{1,2,...,d\}$ and $E$ respectively stand for the set of vertices and the set of edges. We consider a continuous-time jump process $X$ on the vertices of $G$ such that the law of $X$ satisfies the following condition:
\begin{enumerate}
\item[i.] at time $t\le 0$, the local time at each vertex $v\in V$ has a positive initial value $\ell^{(v)}_0$,
\item[ii.] at time $t>0$, given the $\sigma$-field  $\mathcal{F}_t$ generated by $\{X_{s},s\le t\}$, the probability that there is a jump from $X_t$ during $(t,t+h]$ to a neighbour $v$ of $X_t$ (i.e. $\{v,X_t\}\in E$) is given by
$$w\left(\ell^{(v)}_0+\int_0^t \mathbf{1}_{\{X_s=v\}}\rmd s \right)\cdot h+o(h)$$
as $h\to 0$, where $w:[0,\infty)\to(0,\infty)$ is a weight function.
\end{enumerate}

For each vertex $v\in V$, we denote by $L(v,t)=\ell^{(v)}_0+\int_0^t\mathbf{1}_{\{X_s=v\}}\rmd s$  the local time at $v$ up to time $t$ and let 
$$Z_t=\left(\frac{L(1,t)}{\ell_0+t},\frac{L(2,t)}{\ell_0+t},... ,\frac{L(d,t)}{\ell_0+t}\right)$$ stand for the (normalized) occupation measure on $V$ at time $t$, where
$\ell_0=\ell^{(1)}_0+\ell^{(2)}_0+\cdots +\ell^{(d)}_0$.

In our work, we consider the weight function $w(t)=t^{\alpha}$, for some $\alpha>0$. The jump process $X$ is called \textit{strongly vertex-reinforced} if $\alpha>1$, \textit{weakly vertex-reinforced} if $\alpha<1$ or \textit{linearly vertex-reinforced} if $\alpha=1$.

The model of discrete time edge-reinforced random walks (ERRW) was first studied by Coppersmith and Diaconis in their unpublished manuscripts \cite{Coppersmith86} and later the model of discrete time vertex-reinforced random walks (VRRW) was introduced by Pemantle in \cite{Pemantle88} and \cite{Pemantle92}. Several remarkable results about localization of ERRW and VRRW were obtained in \cite{Volkov01}, \cite{Tarres04}, \cite{Volkov06},  \cite{Benaim13} and \cite{Cotar2015}.  Wendelin Werner then proposed a model in continuous time so-called vertex reinforced jump processes (VRJP) whose linear case was first investigated by Davis and Volkov in \cite{Davis02} and \cite{Davis04}. In particular, these authors showed in \cite{Davis04} that linearly VRJP on any finite graph with $d$ vertices is recurrent, i.e. all local times are almost surely unbounded and the normalized occupation measure process converges almost surely to an element in the interior of the $(d-1)$ dimensional standard unit simplex as time goes to infinity. In \cite{Sabot15}, Sabot and Tarr\`es also obtained the limiting distribution of the centred local times process for linearly VRJP on any finite graph and showed that linearly VRJP is
actually a mixture of time-changed Markov jump processes.  Many aspects of linearly VRJP as well as its relations to ERRW and the supersymmetric hyperbolic sigma model have been well studied in recent years (see, e.g. \cite{Collevecchio2009},  \cite{Basdevant2012}, \cite{Disertori14}, \cite{Merkl16}, \cite{Sabot15}, \cite{Sabot15b}, \cite{Sabot2019} \cite{Sabot17}, \cite{Zeng16}, and \cite{Lupu2018}).

The main aim of our paper is to prove that strongly VRJP on a complete  graph $G=(V,E)$  almost surely have an infinite local time at some vertex $v$, while the local times at the remaining vertices remain bounded. The main technique of our proofs is based on the method of stochastic approximation (see, e.g. \cite{Brandiere96, Benaim96, Benaim97, Benaim99}). We organize the present paper as follows. 
In Section \ref{sec:outline}, our main Theorem and an outline of its proof are given. 
In Section \ref{sec:notation}, we give some preliminary notations as well as some results of stochastic calculus being used throughout the paper. We show in Section \ref{sec:Dyn} that the occupation measure process of strongly VRJP on a complete graph is an asymptotic pseudo-trajectory of a flow generated by a vector field. We then prove the convergence towards stable equilibria in Section \ref{sec:convergence} and the non convergence towards unstable equilibria in Section \ref{sec:nonCV}, which yields our above-mentioned main result.

\section{Main result and outline of proof}\label{sec:outline}
The main result of our paper is the following theorem:
\begin{theorem}\label{thm:mainresult}
Assume that $X$ is a strongly VRJP in a complete graph with weight function $w(t)=t^{\alpha}$, for some $\alpha>1$. Then there almost surely exists a vertex such that its local time tends to infinity while the local times at the remaining vertices remain bounded.
\end{theorem}

The main technique to prove this theorem is based on the method of stochastic approximation (see, e.g. \cite{Brandiere96, Benaim96, Benaim97, Benaim99}). The core idea of this method is to describe the asymptotic behaviour of stochastic processes (which are stochastic algorithms in the discrete setting) in terms of the behaviour of ordinary differential equations. When the sample path of a stochastic process is asymptotically close to the solution of an autonomous differential equation, it is reasonable to investigate the relation between the limiting set of this process and the set of equilibria of the associated differential equation.

Let us explain how we make use of this idea in the context of VRJP in a complete graph with super-linear weight function $w(t)=t^{\alpha}$, for some $\alpha>1$. We first make the time change: for $t>0$, set $\tilde{Z}_t=Z_{e^t-\ell_0}$ and $\tilde{X}_t=X_{e^t-\ell_0}$.
The occupation measure $\tilde{Z}$ satisfies the following equation: $$\frac{\rmd\tilde{Z}^i_t}{\rmd t}=-\tilde{Z}^i_t+\mathbf{1}_{\{\tilde{X}_t=i\}}.$$ 
Let now $t$ be a large time. Then, for every fixed $T$, the process $(X_{t+s})_{s\in [0,T]}$ evolves almost like a Markov process with generator $A_t=A(L(\cdot,t))$ (with $A(\lambda \ell)=\lambda^\alpha A(\ell)$). 
It will be remarked in Section \ref{sec:Dyn} that this diffusion has a unique invariant probability $\pi_t=\pi(Z_t)$.
Such properties will allow us to prove Theorem \ref{cvthrm} in which $\tilde{Z}$ is an asymptotic pseudo-trajectory of a semi-flow $\Phi$ generated by the vector field $F(z)=-z+\pi(z)$, i.e. for all $T>0$, the trajectory $(\tilde{Z}_{t+s}:\;s\in [0,T])$ is close as $t\to\infty$ to the trajectory of the semi-flow $(\Phi_s(\tilde{Z}_t):\;s\in [0,T])$.

In Section \ref{sec:convergence}, using Theorem \ref{cvthrm} with the fact that there is a strict Lyapounov function $H$ for the vector field $F$ (i.e. a function such that $\langle F(z),\nabla H(z)\rangle >0$ if and only if $F(z)\ne 0$), we will show that almost surely the limit set of $Z$ is a connected subset of $\mathcal{C}$, the set of equilibria of $F$ (i.e. the set of all $z$ such that $F(z)=0$). Combining with the fact that the set $\mathcal{C}$ is finite, this will prove Theorem \ref{thm:CVeq} stating  the a.s. convergence of $Z$ towards an equilibrium. In Section \ref{sec:convergence}, after having remarked that the stable equilibria of $F$ are Dirac measures $\delta_i$, $i\in V$, we will prove Theorem \ref{thm:localization} asserting that a.s. on the event $Z$ converges to $\delta_i$, $X$ eventually localizes at $i$, i.e. $L(i,\infty)=\infty$ and $\sum_{j\ne i} L(j,\infty)<\infty$.

Finally in Section \ref{sec:nonCV} we will prove Theorem \ref{thm:nonCV_VRJP} wherein  a.s. $Z$ does not converge towards an unstable equilibrium. 
In preparation for the proof of this theorem, we will demonstrate Theorem \ref{THM:nonCV} which is a general non convergence theorem for a class of finite variation c\`adl\`ag processes. 
To do so, we will follow (and correct) arguments from the proof of a theorem by Brandi\`ere and Duflo (see \cite{Brandiere96} or \cite{Duflo1996}), but use a new idea as follows. We will first show that, under additional assumptions, an asymptotic pseudo-trajectory converging towards an unstable equilibrium is attracted exponentially fast towards the unstable manifold of this equilibrium.
This will allow the proof of the non convergence theorem to be reduced to the case where the unstable equilibrium has no stable direction. 
Theorem \ref{thm:nonCV_VRJP} will then permit to conclude the proof of Theorem \ref{thm:mainresult}.

\section{Preliminary notations and remarks}\label{sec:notation}
Throughout this paper, we denote by $\Delta$ and $T\Delta$ respectively the $(d-1)$ dimensional standard unit simplex in $\mathbb{R}^d$ and its tangent space, which are defined by
\begin{align*}
&\Delta=\{z=(z_1,z_2,...,z_d)\in\mathbb{R}^d:z_1+z_2+\cdots +z_d=1, z_j\ge0, j=1,2,\cdots ,d \},\\
&T\Delta=\{z=(z_1,z_2,...,z_d)\in\mathbb{R}^d:z_1+z_2+\cdots +z_d=0\}.
\end{align*}
%Let $e_1=(1,0,\dots,0), e_2=(0,1,0,\dots,0), \dots, e_d=(0,\dots,0,1)$ be the standard orthonormal basis vectors in $\mathbb{R}^d$ and 
Also, let $\|\cdot\|$ and $\left\langle\cdot,\cdot\right\rangle$ denote the Euclidean norm and the Euclidean scalar product in $\mathbb{R}^d$ respectively.
%Let $\mathcal{M}_d$ be the set of all $d\times d$ matrix with real entries.  Also, we define
%$$T\mathcal{M}_d=\left\{A=(A_{ij})_{i,j=1}^d
%\in 
%\mathcal{M}_d \ : \sum_{j=1}^d A_{ij}=0, \forall i=1,2,...,d  \right\}.$$

For a c\`adl\`ag process $Y=(Y_t)_{t\ge0}$, we denote by $Y_{t-}=\lim_{s\to t-}Y_t$ and $\Delta Y_t=Y_t-Y_{t-}$ respectively the left limit and the size of the jump of $Y$ at time $t$. Let $[Y]$ be as usual the \textit{quadratic variation} of the process $Y$. Note that, for a c\`adl\`ag finite variation process $Y$, we have $[Y]_t=\sum_{0<u\le t}(\Delta Y_u)^2$. In the next sections, we will use the following useful well-known results of stochastic calculus (see e.g. \cite{Jacod2003} and \cite{Protter}):  

1. \textbf{Change of variables formula.} (see Theorem 31, p.~78 in \cite{Protter}) Let $A=(A^1_t,A^2_t,\dots,A^d_t)_{t\ge0}$ be a c\`adl\`ag finite variation process in $\mathbb{R}^d$ and let $f:\mathbb{R}^d\to \mathbb{R}$ be a $C^1$ function. Then for $ t\ge 0$,
$$f(A_t)-f(A_0)=\sum_{i=1}^d\int_0^t  \partial_i f(A_{u-}) \rmd A^i_u +\sum_{0<u\le t}\left(\Delta f(A_u)-\sum_{i=1}^d\partial _i f(A_{u-})\Delta A_u^i\right).$$

2. Let $M=(M_t)_{t\ge0}$ be a c\`adl\`ag locally square-integrable martingale with finite variation in $\mathbb{R}$. A well-known result is that if $\E[[M]_t]<\infty$ for all $t$, then $M$ is a true martingale (see e.g. Corollary 3, p.~73 in \cite{Protter}).
The change of variable formula implies that 
$$M_t^2=M_0^2+\int_0^t 2M_{s-}\rmd M_s+[M]_t.$$
Let $\langle M\rangle$ denote the \textit{angle bracket} of $M$, i.e. the unique predictable non-decreasing process such that $M^2-\langle M\rangle$ is a local martingale. Note that $[M]-\langle M\rangle$ is also a local martingale. 

Let $H$ be a locally bounded predictable process and denote by $H\cdot M$ the c\`adl\`ag locally square-integrable martingale with finite variation defined by $(H\cdot M)_t=\int_0^t H_{s} dM_s$. Recall the following rules:
$$\langle H\cdot M\rangle_t=\int_0^t H^2_{s}\rmd\langle M\rangle_s \quad \text{ and }\quad [H\cdot M]_t=\int_0^t H^2_{s}\rmd [M]_s$$
(see Theorem 4.40, p.~48 and the statement 4.54, p.~55 in \cite{Jacod2003}). Recall also that $H\cdot M$ is a square integrable martingale if and only if for all $t>0$, $\E[\langle H\cdot M\rangle_t]<\infty$.

%2. Let $M=(M_t)_{t\ge0}$ be a martingale in $\mathbb{R}$ such that $$M_t=M_0+I_t+\int_0^t K_s ds,$$ where $K$ is an adapted process and $I$ is a pure jump process i.e. $I_t=\sum_{0<s\le t} \Delta I_s$. Let $f:\mathbb{R}\to\mathbb{R}$ be a $C^1$ function. Then
%$$f(M_t)= f(M_0)+\int_0^t f'(M_s)K_sds + \sum_{0<s\le t}\Delta f(M_s).$$
%In particular,
%$$M_t^2=M_0^2+ 2\int_0^t M_sK_s ds + \sum_{0<s\le t}\Delta M^2_s.$$
%Let $N$ be a process defined by $$N_t=\int_0^t H_s dM_s,$$ where $H$ a bounded predictable process. Then $$N_t=\sum_{0<s\le t} H_s\Delta I_s + \int_0^t H_s K_s ds.$$
%It is also a martingale and $$[N]_t=\int_0^t H_s^2 d [M]_s.$$ 
%Furthermore, 
%$$\E[N_t^2]=\E[\sum_{0<s\le t} (\Delta N_s)^2]=\E[\sum_{0<s\le t} H_s^2(\Delta I_s)^2]=\E[\int_0^t H_s^2 d [M]_s]=\E[[N]_t],$$
%where we recall that $[M]_t=\sum_{0<s\le t} (\Delta I_s)^2$.\\
%

%3. Let $\{X_t\}$ be a semi-martingale and $f$ be a $C^2$ function, we have the generalized Ito formula
%$$  \begin{array}{rl} \displaystyle f(X_t) =&\displaystyle f(X_s)+\int_s^t f^\prime(X_{u-})\,dX_u + \frac{1}{2}\int_0^t f^{\prime\prime}(X_{u-})\,d[X]_u\smallskip\\ &\displaystyle +\sum_{s<u\le t}\left(\Delta f(X_u)-f^\prime(X_{u-})\Delta X_u-\frac{1}{2}f^{\prime\prime}(X_{u-})\Delta X_u^2\right). \end{array} $$

3. \textbf{Integration by part formula.} (see Corollary 2, p.~68 in \cite{Protter}) Let $X=(X)_{t\ge0}$ and $Y=(Y)_{t\ge0}$ be two c\`adl\`ag finite variation processes in $\mathbb{R}$. Then for $t\ge s\ge 0$,
$$X_tY_t-X_sY_s=\int_s^t X_{u-}\rmd Y_u+\int_s^t Y_{u-}\rmd X_u+[X,Y]_t-[X,Y]_s,$$
where we recall that $[X,Y]$ is the \textit{covariation} of $X$ and $Y$, computed as $[X,Y]_t=\sum_{0<u\le t}\Delta X_u \Delta Y_u$.

4. \textbf{Doob's maximal inequality.} (see Theorem 20, p.~11 in \cite{Protter})  Let $X=(X)_{t\ge0}$ be a c\`adl\`ag martingale adapted to a filtration $(\mathcal F_t)_{t\ge0}$. Then for any $p>1$ and $t\ge s\ge 0$, 
$$ \E[\sup_{s\le u\le t}|X_u|^p \big|\mathcal F_s ]\le\left(\frac{p}{p-1}\right)^p{\E}[\vert X_t\vert^p\big|\mathcal F_s].$$

5. \textbf{Burkholder-Davis-Gundy inequality.} (see Theorem 48, p.~193 in \cite{Protter}) Let $X=(X)_{t\ge0}$ be a c\`adl\`ag martingale adapted to a filtration $(\mathcal F_t)_{t\ge0}$ such that $X_0=0$. For each $1\le p<\infty$ there exist positive constants $c_p$ and $C_p$ depending on only $p$ such that 
$$\displaystyle c_p{\E}\left[ [X]^{p/2}_t\big|\mathcal F_s\right]\le{\E}\left[\sup_{s\le u\le t}|X_u|^p \big|\mathcal F_s\right]\le C_p{\E}\left[ [X]^{p/2}_t \big|\mathcal F_s\right].$$

\section{Dynamics of the occupation measure process\label{sec:Dyn}}

We study in this section the dynamics of the occupation process of VRJP on a complete graph with weight function $w(t)=t^{\alpha},\ \alpha>0$. In particular, we show in Theorem \ref{cvthrm} below that, after a time scaling, the occupation measure process is asymptotically close to the unique solution of an autonomous system of ordinary differential equations. Our approach is inspired by the theory of asymptotic pseudo-trajectories and stochastic approximation techniques introduced in \cite{Benaim99}.

For $t>0$ which is not a jumping time of $X$, we have
\begin{equation}\label{ode}
\frac{\rmd Z_t}{\rmd t}=\frac{1}{\ell_0+t}\left(-Z_t+I[{X_t}]\right),
\end{equation}
where for each matrix $M$, $M[j]$ is the $j$-th row vector of $M$ and $I$ is as usual
the identity matrix. Observe that the process $Z=(Z_t)_{t\ge0}$ always takes values in the interior of the standard unit simplex $\Delta$. 

For fixed $t\ge0$, let $A_t$ be the $d$-dimensional infinitesimal generator matrix such that the $(i,j)$ element is defined by
$$%(A_t)_{i,j}
A^{i,j}_t:=\left\lbrace\begin{matrix}\mathbf{1}_{(i,j)\in E} w_t^{(j)}, \ \ \ \ \  \ i\neq j; \\ \displaystyle - \sum_{k\in V,(k,i)\in E} w_t^{(k)},  i=j,\end{matrix}\right.$$ 
where we have set $w^{(j)}_t=w(L(j,t))=L(j,t)^{\alpha}$ for each $j\in V$. Also, let $w_t=w^{(1)}_t+w^{(2)}_t+\cdots +w^{(d)}_t$. 
Note that 
$$\pi_t:=\left(\frac{w^{(1)}_t}{w_t},\frac{w^{(2)}_t}{w_t},\cdots ,\frac{w^{(d)}_t}{w_t}\right)$$
is the unique invariant probability measure of $A_t$ in the sense that $\pi_tA_t=0$. Since $\pi_t$ can be rewritten as a function of $Z_t$, we will also use the notation $\pi_t=\pi(Z_t)$, where we define the function $\pi: \Delta\to \Delta$, such that for each $z=(z_1,z_2,...,z_d)\in \Delta$,
$$\pi(z)=\left( \frac{z_1^{\alpha}}{z_1^{\alpha}+\cdots +z_d^{\alpha}},\cdots , \frac{z_d^{\alpha}}{z_1^{\alpha}+\cdots +z_d^{\alpha}} \right).$$
Now we can rewrite the equation \eqref{ode} as 
\begin{equation}\label{ode2}\frac{\rmd Z_t}{\rmd t}=\frac{1}{\ell_0+t}(-Z_t+\pi_t) + \frac{1}{\ell_0+t}(I[X_t]-\pi_t).\end{equation}
Changing variable $\ell_0+t=e^u$ and denoting $\tilde{Z}_u=Z_{e^u-\ell_0}$ for $u>0$, we can transform the equation \eqref{ode2} as
\begin{equation*}\frac{\rmd \tilde{Z}_u}{\rmd u}=-\tilde{Z}_u+\pi(\tilde{Z}_u) + (I[X_{e^u-\ell_0}]-\pi_{e^u-\ell_0}).\end{equation*}
Taking integral of both sides, we obtain that
\begin{equation}\label{ode3}\tilde{Z}_{t+s}-\tilde{Z}_{t}=\int_t^{t+s}\left(-\tilde{Z}_u+\pi(\tilde{Z}_u) \right) \rmd u +\int_{e^t-\ell_0}^{e^{t+s}-\ell_0} \frac{ I[X_u]-\pi_u}{\ell_0+u}\rmd u.\end{equation}

%For each $z=(z_1,z_2,\cdots, z_d)\in \Delta$, define the matrix $\tilde A(z)$ such that for $i,j\in\{1,2,...,d\}$,
%$$\tilde A_{i,j}(z) =\left\lbrace\begin{matrix}\mathbf{1}_{(i,j)\in E} z_{j}^{\alpha}, \ \ \ \ \  \ i\neq j; \\ \displaystyle - \sum_{k\in V,(k,i)\in E} z_{k}^{\alpha},  i=j,\end{matrix}\right.$$ 
%For $t\ge 0$, let $\tilde{A}_t=\tilde{A}(Z_t)$. Note that $\tilde{A}_t=\frac{1}{w_t}A_t$. Also, for each $z$ in the interior of $\Delta$ (denoted by $\text{int}(\Delta)$), let us define the matrix  
%\begin{equation}Q(z)=\int_0^\infty e^{t\tilde{A}(z)}\left(I-\Pi(z)\right) dt,
%\end{equation} where we define $\Pi_{ij}(z)=\pi_j(z)$ for each $i,j \in\{1,2,\cdots,d\}$ and $z\in \Delta$. Observe that $Q: \text{int}(\Delta)\to T\mathcal{M}_d$ is  a matrix function satisfying the following equation 
%\begin{equation}\Pi(z)-I=\tilde{A}(z)Q(z),\ \forall z\in \text{int}(\Delta).
%\end{equation}
%Furthermore, for $t\ge 0$, 
%\begin{equation}\|Q(Z_t)\|<k_1 (t+\ell_0)^{\alpha}, \left\|\frac{\partial}{\partial z_i}Q(Z_t)\right\|<k_2(t+\ell_0)^{\alpha},
%\end{equation}  
%where $k_1,k_2$ are positive constant.
%
%In the case of complete graph, we note that
%$Q(z)=I-\Pi(z)$. Thus, $Q(z)$ can be defined for all $z\in \Delta$.  We also have the following boundedness
%\begin{equation}\sup_{z\in \Delta}\|Q(z)\|<2,\ \sup_{z\in \Delta}\left\|\frac{\partial}{\partial z_i}Q(z)\right\|<2.
%\end{equation}  

Let us fix a function $f:\{1,\dots,d\}\to\mathbb{R}$. For $t>0$, define $A_tf:\{1,\dots,d\}\to\mathbb{R}$ by $A_tf(i)=\sum_j A_t^{i,j}f(j)$ and define the process $M^f$ by 
$$M^f_t=f(X_t)-f(X_0)-\int_0^t A_sf(X_s) \rmd s.$$ 
\begin{lemma}
The process $M^f$ is a martingale, with $[M^f]_t=\sum_{0<s\le t} (\Delta f(X_s))^2$ and 
\begin{equation}\label{eq:anglebracket}
\langle M^f\rangle_t = \int_0^t \big(A_sf^2(X_s)-2f(X_s)A_sf(X_s)\big) \rmd s.
\end{equation} 
\end{lemma}
\begin{proof}
Let us first prove that $M^f$ is a martingale. For small $h>0$, we have
\begin{align*}
\E[f(X_{t+h})-f(X_t)|\mathcal{F}_t]
&= \sum_{j\sim X_t} (f(j)-f(X_t)) \P[X_{t+h}=j|\mathcal{F}_t]\\
&= \sum_{j\sim X_t} (f(j)-f(X_t)) w^{(j)}_t.h+o(h)\\
&= A_tf(X_t).h + o(h).
\end{align*}

Let us fix $0<s<t$ and define $t_j=s+j(t-s)/n$ for $j=0,1,\dots,n$. Note that  
\begin{align*} \E\left[ f(X_{t})-f(X_s)|\ \mathcal{F}_s  \right] & =\E\left[\left. \sum_{j=1}^n \E[f(X_{t_{j}})-f(X_{t_{j-1}})\ |\ \mathcal{F}_{t_{j-1}} ]\ \right|\ \mathcal{F}_s \right]\\
& =\E\left[ \left.\sum_{j=1}^n A_{t_{j-1}}f(X_{t_{j-1}})(t_j-t_{j-1}) +n\cdot o\left(\frac{t-s}{n}\right) \ \right| \ \mathcal{F}_s  \right].
\end{align*}
Since the left hand side is independent on $n$, using Lebesgue's dominated convergence theorem and taking the limit of the random sum under the expectation sign on the right hand side, we obtain that
$$\E\left[ f(X_{t})-f(X_s)|\ \mathcal{F}_s  \right)=\E\left[\int_s^t A_uf(X_u) \rmd u\ | \ \mathcal{F}_s \right].$$ Thus, $\E[ M^f_t|\ \mathcal{F}_s  ]=M_s$.

To prove \eqref{eq:anglebracket}, 
we calculate (to simplify the calculation, we will suppose that $f(X_0)=0$).
\begin{align*}
(M^f_t)^2
=&\; f^2(X_t)- 2\left(M^f_t+\int_0^t A_sf(X_s)\rmd s\right)\int_0^t A_sf(X_s)\rmd s + \left(\int_0^t A_sf(X_s)\rmd s \right)^2\\
=&\; M^{f^2}_t+\int_0^t A_sf^2(X_s)\rmd s - 2M^f_t\int_0^t A_sf(X_s)\rmd s - \left(\int_0^t A_sf(X_s)\rmd s \right)^2\\
=&\; N_t + \int_0^t A_sf^2(X_s)\rmd s\\
&\quad - 2\int_0^t M^f_s A_sf(X_s)\rmd s - 2\int_0^t A_sf(X_s)\left(\int_0^s A_uf(X_u)\rmd u\right)\rmd s \\
=&\; N_t + \int_0^t A_sf^2(X_s)\rmd s - 2\int_0^t f(X_s) A_sf(X_s)\rmd s,
\end{align*}
where the process $N$, defined by $N_t:=M_t^{f^2}-2\int_{0}^t\left(\int_0^sA_uf(X_u)\rmd u\right)\rmd M^f_s $, is a martingale.
The lemma is proved.
\end{proof}

Let $M$ be the process in $\mathbb{R}^d$ defined by
$$M_t=I[X_t]-\int_0^t A_s[X_s]\rmd s\quad \text{for } t\ge0 .$$
Then for each $j$, $M^j$ is a martingale since $M^j=M^{\delta_j}$, with $\delta_j$ defined by $\delta_j(i)=1$ if $i=j$ and $\delta_j(i)=0$ if $i\neq j$. 
%For convenience, set $A_t=A_t[X_t]$ and note that if $X_t=i$, we have $A_t^j=A^{i,j}_t$. 
We also have that 
\begin{equation}\label{eq:crochetMj}
\langle M^j\rangle_t = \int_0^t \Lambda^j_s \rmd s,
\end{equation}
with $\Lambda^j$ defined by
\begin{equation}\label{eq:defLambdaj}
\Lambda^j_t= \left\lbrace \begin{array}{ll}
w^{(j)}_t & \text{ if }  \quad X_t\sim j,\\
\sum_{k\sim X_t} w^{(k)}_t  &\text{ if }\quad   X_t=j, \\
0 & \text{ otherwise. } \\
\end{array}
\right.
\end{equation}

%\textsc{Take these modifications in the following. I will not look today at the rest of this section and at section 4.}

\begin{lemma}\label{noise}
Assume that $G=(V,E)$ is a complete graph and $w(t)=t^{\alpha}$ with $\alpha>0$. Then almost surely 
\begin{equation}
\label{bound}\lim_{t\to\infty} \sup_{1\le c\le C} \left\|\int_{t-\ell_0}^{ct-\ell_0} \frac{I[X_s]-\pi_s}{\ell_0+s}\rmd s\right\| =0
\end{equation}
for each $C>1$.
\end{lemma}
\begin{proof}  
Note that, for $t\ge 0$,
$$\pi_t-\displaystyle I[X_t]=\frac{1}{w_t}A_t[X_t].$$
Using the integration by part formula, we obtain the following identity for each $c\in [1,C]$
\begin{align*}
\int_{t-\ell_0}^{ct-\ell_0} \frac{\pi_s-I[X_s]}{\ell_0+s}\rmd s&=\int_{t-\ell_0}^{ct-\ell_0} A_s[X_s]\frac{\rmd s}{(\ell_0+s)w_s}\\
&= \left(\frac{I[X_{ct-\ell_0}]}{ctw_{ct-\ell_0}}-\frac{I[X_{t-\ell_0}]}{tw_{t-\ell_0}}\right)\\
& - \int_{t-\ell_0}^{ct-\ell_0} I[X_s] \frac{\rmd}{\rmd s}\left(\frac{1}{(s+\ell_0)w_s}\right)\rmd s\\
& - \int_{t-\ell_0}^{ct-\ell_0} \frac{\rmd M_s}{(s+\ell_0)w_s}. 
 \end{align*}

Observe that for some positive constant $k$, $w_s\ge k s^\a$ (which is easy to prove, using the fact that $L(1,t)+L(2,t)+\cdots +L(d,t)=\ell_0+t$). We now estimate the terms in the right hand side of the above-mentioned identity. In the following, the positive constant $k$ may change from lines to lines and only depends on $C$ and $\ell_0$.
First,
\begin{equation}\label{first}\left\|\frac{I[X_{ct-\ell_0}]}{ctw_{ct-\ell_0}}-\frac{I[X_{t-\ell_0}]}{tw_{t-\ell_0}}\right\| \le    k/ {t^{\a+1}}.
\end{equation}
Second, for $s\in [t,ct]$ which is not a jump time, we have
\begin{align*}
\frac{\rmd}{\rmd s}\left(\frac{1}{(\ell_0+s)w_s}\right) = & -\left(\frac{1}{(\ell_0+s)^2w_s}+\frac{1}{(\ell_0+s)w^2_s}\frac{\rmd w_s}{\rmd s}\right).
\end{align*}
When $s$ is not a jump time, it is easy to check that $\left|\frac{\rmd w_s}{\rmd s}\right|\le \a (\ell_0+s)^{\a-1}$. Therefore, for $s\in [t,ct]$ which is not a jump time,
$$\left|\frac{\rmd}{\rmd s}\left(\frac{1}{(\ell_0+s)w_s}\right)\right| \le  k/s^{2+\a}$$
and thus,
\begin{equation}\label{second}\left\|\int_{t-\ell_0}^{ct-\ell_0} I[X_s] \frac{\rmd}{\rmd s}\left(\frac{1}{(\ell_0+s)w_s}\right) \rmd s\right\| \le  k/t^{\a+1}.
\end{equation}
And at last (using Doob's inequality), for $i\in\{1,2,\cdots ,d\}$,
\begin{eqnarray*}
\E\left[\sup_{1\le c\le C} \left|\int_{t-\ell_0}^{ct-\ell_0} \frac{\rmd M^i_s}{(\ell_0+s)w_s}\right|^2\right]
&\le& 4\ \E\left[\left(\int_{t-\ell_0}^{Ct-\ell_0} \frac{\rmd M^i_s}{(\ell_0+s)w_s}\right)^2\right].
\end{eqnarray*}
Observe that in our setting, for $i\in\{1,2,\cdots ,d\}$, $(\Delta I^i_s)^2=1$ if $s$ is a jump time between $i$ and another vertex. Thus $[M^1]_t+[M^2]_t+\cdots +[M^d]_t$ is just twice the number of jumps up to time $t$ of $X$. So,  for $i \in\{1,2,\cdots ,d\}$,
\begin{eqnarray*}
\E\left[\left(\int_{t-\ell_0}^{Ct-\ell_0} \frac{\rmd M^i_s}{(\ell_0+s)w_s}\right)^2\right] 
&=& \E\left[\int_{t-\ell_0}^{Ct-\ell_0} \frac{\rmd [M^i]_s}{(\ell_0+s)^2 w_s^2}\right]\\
&\le& \frac{k}{t^{2(\a+1)}} \E\left[ [M^i]_{Ct-\ell_0}-[M^i]_{t-\ell_0}\right]\\
&\le& \frac{k}{t^{2(\a+1)}} (Ct)^\a (C-1)t,
\end{eqnarray*}
where in the last inequality,  we have used the fact that the number of jumps in $[t-\ell_0,Ct-\ell_0]$ is dominated by the number of jumps of a Poisson process with constant intensity $(C t)^\a$ in $[t-\ell_0,Ct-\ell_0]$. Therefore,
\begin{eqnarray}\label{third}
\E\left[\sup_{1\le c\le C} \left\|\int_{t-\ell_0}^{ct-\ell_0} \frac{\rmd M_s}{(\ell_0+s)w_s}\right\|^2\right]
&\le& \frac{k}{t^{\alpha+1}}.
\end{eqnarray}
From (\ref{first}), (\ref{second}), (\ref{third}) and by using Markov's inequality, we have
\begin{equation}\label{Markov}\P\left[\sup_{1\le c\le C} \left\|\int_{t-\ell_0}^{ct-\ell_0} \frac{I[X_s]-\pi_s}{\ell_0+s}\rmd s\right\| \ge \frac{1}{t^{\gamma}}\right]\le  \frac{k}{t^{\alpha+1-2\gamma}}
\end{equation}
for every $0<\gamma\le \frac{\alpha+1}{2}$.
Using the Borel-Cantelli lemma, we thus obtain
$$\limsup_{n\to\infty}\sup_{1\le c\le C} \left\|\int_{C^n-\ell_0}^{cC^n-\ell_0} \frac{I[X_s]-\pi_s}{\ell_0+s}\rmd s\right\| =0.$$ 
Moreover, for $C^n\le t\le C^{n+1}$, we have
\begin{align*}  \sup_{1\le c\le C} \left\|\int_{t-\ell_0}^{ct-\ell_0} \frac{I[X_s]-\pi_s}{\ell_0+s}\rmd s\right\|   & \le \left\|\int_{C^n-\ell_0}^{t-\ell_0} \frac{I[X_s]-\pi_s}{\ell_0+s}\rmd s\right\|  +\sup_{1\le c\le C} \left\|\int_{C^n-\ell_0}^{\min(ct,C^{n+1})-\ell_0} \frac{I[X_s]-\pi_s}{\ell_0+s}\rmd s\right\|\\
&+ \sup_{1\le c\le C} \left\|\int^{\max(ct,C^{n+1})-\ell_0}_{C^{n+1}-\ell_0} \frac{I[X_s]-\pi_s}{\ell_0+s}\rmd s\right\|
\\
&  \le 2\sup_{1\le c\le C} \left\|\int_{C^n-\ell_0}^{cC^n-\ell_0} \frac{I[X_s]-\pi_s}{\ell_0+s}\rmd s\right\| \\
&+\sup_{1\le c\le C} \left\|\int_{C^{n+1}-\ell_0}^{cC^{n+1}-\ell_0} \frac{I[X_s]-\pi_s}{\ell_0+s}\rmd s\right\|.
\end{align*}
This inequality immediately implies (\ref{bound}). 
\end{proof}

From now on, we always assume that $w(t)=t^{\alpha}$, $\alpha>1$ and $G=(V,E)$ is a complete graph. Let us define the vector field $F:\Delta\to T\Delta $ such that $F(z)=-z +\pi(z)$ for each $z\in \Delta$. We also remark that for each $z=(z_1,z_2,\cdots,z_d)\in \Delta$,
\begin{align}\label{vecF}F(z)=\left(-z_1+ \frac{z_1^{\alpha}}{z_1^{\alpha}+\cdots +z_d^{\alpha}},\cdots ,-z_d+ \frac{z_d^{\alpha}}{z_1^{\alpha}+\cdots +z_d^{\alpha}} \right).\end{align}

A continuous map $\Phi: \mathbb{R}_+\times \Delta \to \Delta$ is called a \textit{semi-flow} if $\Phi(0,\cdot):\Delta \to \Delta$ is the identity map and $\Phi$ has the semi-group property, i.e. $\Phi({t+s},\cdot)=\Phi(t,\cdot)\circ \Phi(s,\cdot)$ for all $s,t\in \mathbb{R}_{+}$. 

Now for each $z^0\in \Delta$, let $\Phi_t(z^0)$ be the solution of the differential equation  
\begin{equation}\label{odeF}\left\lbrace
\begin{array}{ll}
\displaystyle 
\frac{\rmd}{\rmd t}z(t)  = & F(z(t)),\ t>0;\\
 z(0)\ \ \ = & z^0.
\end{array}
\right.
\end{equation}
Note that $F$ is Lipschitz. Thus the solution $\Phi_t(z^0)$ can be extended for all $t\in \mathbb{R}_+$ and $\Phi:\mathbb{R}_+\times \Delta \to \Delta$ defined by $\Phi(t,z)=\Phi_t(z)$ is a semi-flow.

\begin{theorem}\label{cvthrm} $\tilde{Z}$ is an asymptotic pseudo-trajectory of the semi-flow $\Phi$, i.e.
for all $T>0$,
\begin{equation}\label{pseu}
\lim_{t\to\infty} \sup_{0\le s\le T} \left\| \tilde{Z}_{t+s}- \Phi_s(\tilde{Z}_{t})\right\|=0.\ \text{a.s}.
\end{equation}
Furthermore, $\tilde{Z}$ is an -$\frac{\alpha+1}{2}$-asymptotic pseudo-trajectory, i.e. for 
\begin{equation}\label{pseu2}\limsup_{t\to\infty}\frac{1}{t}\log\left( \sup_{0\le s\le T} \| \tilde{Z}_{t+s}-\Phi_s(\tilde{Z}_{t})\| \right)\le -\frac{\alpha+1}{2} \ \text{a.s}.
\end{equation}
\end{theorem}
\begin{proof}
From the definition of $\Phi$, we have
$$\Phi_s(\tilde{Z}_{t})-\tilde{Z}_{t}=\int_{0}^{s}F(\Phi_u(\tilde{Z}_{t}))\rmd u.$$
Moreover, from \eqref{ode3}
$$\tilde{Z}_{t+s}-\tilde{Z}_{t}=\int_0^{s}F(\tilde{Z}_{t+u}) \rmd u +\int_{e^t-\ell_0}^{e^{t+s}-\ell_0} \frac{ I[X_u]-\pi_u}{\ell_0+u}\rmd u.$$
Subtracting both sides of the two above identities, we obtain that
$$\tilde{Z}_{t+s}-\Phi_s(\tilde{Z}_{t})=\int_0^{s}\left( F(\tilde{Z}_{t+u}) - F(\Phi_u(\tilde{Z}_{t}))\right) \rmd u+\int_{e^t-\ell_0}^{e^{t+s}-\ell_0} \frac{ I[X_u]-\pi_u}{\ell_0+u}\rmd u.$$
Observe that $F$ is Lipschitz, hence
$$ \| \tilde{Z}_{t+s}-\Phi_s(\tilde{Z}_{t})\| \le K \int_0^{s}\|\tilde{Z}_{t+u}-\Phi_u(\tilde{Z}_{t})\|\rmd u +\left\|\int_{e^t-\ell_0}^{e^{s+t}-\ell_0} \frac{I[X_u]-\pi_u}{\ell_0+u}\rmd u\right\|,$$
where $K$ is the Lipschitz constant of $F$. Using Gr\"onwall's inequality, we thus have 
\begin{equation}\label{gron}\| \tilde{Z}_{t+s}-\Phi_s(\tilde{Z}_{t})\|\le \sup_{0\le s \le T}\left\|\int_{e^t-\ell_0}^{e^{s+t}-\ell_0} \frac{I[X_u]-\pi_u}{\ell_0+u}\rmd u\right\| e^{K s}.\end{equation}
On the other hand, from Lemma \ref{noise}, we have
\begin{equation}\label{noise2}
\lim_{t\to\infty} \sup_{0\le s\le T} \left\|\int_{e^t-\ell_0}^{e^{s+t}-\ell_0} \frac{I[X_u]-\pi_u}{\ell_0+u}\rmd u\right\| =0. \ \ \text{a.s.}
\end{equation}
The inequality (\ref{gron}) and (\ref{noise2}) immediately imply (\ref{pseu}). 

%On the other hand, from the inequality (\ref{gron}) in the proof of Lemma 3, we obtain that for $T>0$,
%\begin{align}\label{logbound1}\sup_{0\le s\le T}\| \tilde{Z}_{t+s}-\Phi_s(\tilde{Z}_{t})\| &\le k \sup_{0\le s\le T} \left\|\int_{e^t}^{e^{s+t}} \frac{I[X_u]-\pi_u}{\ell_0+u}du\right\|,
%\end{align}
%where $k$ is some positive constant depending on $T$ and may change from lines to lines.
We now prove the second part of the theorem. From (\ref{Markov}), we have
$$\P\left[\sup_{0\le s\le T}  \left\|\int_{e^t}^{e^{s+t}} \frac{I[X_u]-\pi_u}{\ell_0+u}\rmd u\right\| \ge e^{-\gamma t}\right] \le k e^{-(\alpha+1-2\gamma)t},$$
for every $0<\gamma\le\frac{\alpha+1}{2}$. By Borel-Cantelli lemma, it implies that 
$$\limsup_{n\to\infty}\frac{1}{nT}\log\left( \sup_{0\le s\le T} \left\|\int_{e^{nT}}^{e^{s+nT}} \frac{I[X_u]-\pi_u}{\ell_0+u}\rmd u\right\| \right)\le -\gamma\ \ \text{a.s}.$$
and therefore that (taking $\gamma\to \frac{\alpha+1}{2}$) %by the continuity from the right of cumulative  distribution functions, we obtain that
$$\limsup_{n\to\infty}\frac{1}{nT}\log\left( \sup_{0\le s\le T} \left\|\int_{e^{nT}}^{e^{s+nT}} \frac{I[X_u]-\pi_u}{\ell_0+u}\rmd u\right\| \right)\le -\frac{\alpha+1}{2} \ \text{a.s}.$$
Note that for $nT\le t\le (n+1)T$ and $0\le s\le T$,
\begin{align*}
\left\|\int_{e^t}^{e^{s+t}} \frac{I[X_u]-\pi_u}{\ell_0+u}\rmd u\right\| & \le 2\sup_{0\le s\le T}\left\|\int_{e^{nT}}^{e^{s+nT}} \frac{I[X_u]-\pi_u}{\ell_0+u}\rmd u\right\|\\
& +\sup_{0\le s\le T}\left\|\int_{e^{(n+1)T}}^{e^{s+(n+1)T}} \frac{I[X_u]-\pi_u}{\ell_0+u}\rmd u\right\|.
\end{align*}
Therefore,
\begin{align}\label{logbound2}\limsup_{t\to\infty}\frac{1}{t}\log\left( \sup_{0\le s\le T} \left\|\int_{e^{t}}^{e^{s+t}} \frac{I[X_u]-\pi_u}{\ell_0+u}\rmd u\right\| \right)\le -\frac{\alpha+1}{2} \ \text{a.s}.
\end{align}
Finally, (\ref{pseu2}) is obtained from (\ref{gron}) and (\ref{logbound2}).
\end{proof}
\section{Convergence to equilibria}\label{sec:convergence}
Let 
$$\mathcal{C}=\{z\in \Delta : F(z)=0\}$$
stand for the \textit{equilibria set} of the vector field $F$ defined in (\ref{vecF}). We say an equilibrium $z\in \mathcal{C}$ is (linearly) \textit{stable} if all the eigenvalues of $DF(z)$, the Jacobian matrix of $F$ at $z$, have negative real parts. If there is one of its eigenvalues having a positive real part, then it is called (linearly) \textit{unstable}.   

Observe that $\mathcal{C}=\mathcal{S}\cup \mathcal U$, where we define
$$\mathcal S=\{e_1=(1,0,0,\cdots ,0), e_2=(0,1,0,\cdots ,0),\cdots ,e_d= (0,0,\cdots ,0,1)\}$$ as the set of all stable equilibria and 
$$\mathcal U=\{ z_{j_1,j_2,\cdots, j_k} : 1\le j_1<j_2<\cdots <j_k\le d, k=2,\cdots ,d\}$$ as the set of all unstable equilibria, where $z_{j_1,j_2,\cdots, j_k}$ stands for the point $z=(z_1,\cdots, z_d)\in\Delta$ such that $z_{j_1}=z_{j_2}=\cdots =z_{j_k}=\frac{1}{k}$ and all the remaining coordinates are equal to 0. 

Indeed, for each $z\in \mathcal S$, we have that $DF(z)=-I$. Moreover, $$DF\left(\frac{1}{d},\frac{1}{d},\cdots ,\frac{1}{d}\right)=(\alpha-1) I -\frac{\alpha }{d} N,$$ where 
%$N=\left(
%\begin{array}{cccc}
%1 & 1 & \cdots  & 1 \\
%1 & 1 & \cdots  & 1 \\
% \vdots  & \vdots  & \cdots  & \vdots  \\
%1 & 1 & \cdots  & 1 \\
%\end{array}
%\right)$
$N$ is the matrix such that $N_{m,n}=1$ for all $m,n$ 
and $DF(z_{j_1,j_2,\cdots,j_k})= (D_{m,n})$ where $$D_{m,n}=   
\left\lbrace 
\begin{array}{ll}
(\alpha-1)- \frac{\alpha}{k} &\text{if } m=n\in\{j_i:\;i=1,\cdots,k\}; 
\\
-\frac{\alpha}{k} & \text{if } m\neq n, \hbox{ with } \{m,n\} \subset\{j_i:\;i=1,\cdots,k\};\\
-1 & \text{if } m=n\not\in\{j_i:\;i=1,\cdots,k\};\\
0 & \text{if } m\neq n, \hbox{ with } \{m,n\} \not\subset\{j_i:\;i=1,\cdots,k\}.
\end{array} \right.
$$
Therefore, we can easily compute that for each $z\in \mathcal U$, the eigenvalues of $DF(z)$ are $-1$ and $\alpha-1$, having respectively multiplicity $d-k+1$ and $k-1$. 
  
\begin{theorem}\label{thm:CVeq} $Z_t$ converges almost surely to a point in $\mathcal{C}$ as $t\to\infty$.   
\end{theorem}
\begin{proof}
Consider the map $H:\Delta\to \mathbb{R}$ such that
$$H(z)=z_1^{\alpha}+z_2^{\alpha}+\cdots +z_n^{\alpha}.$$
Note that $H$ is a strict Lyapounov function of $F$, i.e $\langle \nabla H(z), F(z)\rangle$ is positive for all $z\in \Delta\setminus \mathcal{C}$. Indeed, we have
\begin{align*} \langle \nabla H(z),F(z)\rangle  & =\displaystyle \sum_{i=1}^d \alpha z_i^{\alpha-1} \left(-z_i +\frac{z_i^{\alpha}}{\sum_{j=1}^d  z_j^{\alpha}} \right)\\
&=\alpha    \left(-\sum_{i=1}^d z_i^{\alpha} +\frac{\sum_{i=1}^d z_i^{2\alpha-1}}{\sum_{i=1}^d  z_i^{\alpha}} \right)\\
\ & \displaystyle  =\frac{\alpha}{H(z)}    \left(-\left( \sum_{i=1}^d z_i^{\alpha}\right)^2 +\sum_{i=1}^d z_i^{2\alpha-1}\sum_{i=1}^dz_i \right)\\
\ & \displaystyle =\frac{\alpha}{H(z)} \sum_{1\le i<j\le d} z_iz_j \left( z_i^{\alpha-1}-z_j^{\alpha-1}\right)^2.
\end{align*}
For $z\in \Delta\setminus \mathcal{C}$, there exist distinct indexes $j_1, j_2\in \{1,2,...,d\}$ such that $z_{j_1}, z_{j_2}$ are positive and $z_{j_1}\neq z_{j_2}$. Therefore,
 $$\langle \nabla H(z),F(z)\rangle \ge \frac{\alpha}{H(z)}  z_{j_1}z_{j_2} \left( z_{j_1}^{\alpha-1}-z_{j_2}^{\alpha-1}\right)^2>0.$$
 
Let $$L(Z)=\bigcap_{t\ge0}\overline{Z([t,\infty))}$$
be limit set of $Z$. Since $\tilde{Z}$ is  an asymptotic pseudo-trajectory of $\Phi$, by Theorem 5.7 and Proposition 6.4 in \cite{Benaim99}, we can conclude that  $L(Z)=L(\tilde{Z})$ is a connected subset of $\mathcal{C}$. Moreover, $\mathcal{C}$ is actually an isolated set and this fact implies the almost sure convergence of ${Z_t}$ toward an equilibrium $z\in\mathcal{C}$ as $t\to\infty$.
\end{proof}

\begin{lemma} \label{cvrate}  Let $z^*$ be a stable equilibrium. Then for each small $\epsilon>0$ there exists $\delta_{\epsilon}>0$ such that $z^*$ attracts exponentially $B_{\delta_{\epsilon}}(z^*):=\left\{z\in\Delta : \|z-z^*\|<\delta_{\epsilon} \right\}$ at rate $-1+\epsilon$, i.e.
$$\|\Phi_s(z)-z^*\|\le e^{-(1-\epsilon)s}\|z-z^*\|$$
for all $s>0$ and $z\in B_{\delta_{\epsilon}}(z^*)$.
\end{lemma}
\begin{proof}
We observe that $$F(z)=(z-z^*).DF(z^*)^T+R(z-z^*),$$
where we have set $$R(y)=y.\left(\int_0^1DF(ty+z^*)^T\rmd t-DF(z^*)^T\right).$$ Note that $\|R(y)\| \le k \|y\|^{1+\beta},$ where $\beta=\min(1,\alpha-1)$ and $k$ is some positive constant. Therefore, we can transform the differential equation \eqref{odeF} to the following integral form
$$z(t)-z^*= (z(0)-z^*)e^{t DF(z^*)^T}+\int_0^t R(z(s)-z^*)e^{(t-s)DF(z^*)^T}\rmd s. $$
Note that for $z^*\in \mathcal{S}$, we have $DF(z^*)=-I$. Therefore,
$$\|z(t)-z^*\|\le e^{-t}\| z(0)-z^*\|+\int_0^t e^{-(t-s)}\|R(z(s)-z^*)\|\rmd s.$$
For each small $\epsilon>0$, if $\|z(s)-z^*\|\le\left(\frac{\epsilon}{k}\right)^{1/\beta}$ for all $0\le s\le t$, then
$$e^{t}\|z(t)-z^*\|\le \|z(0)-z^*\|+\epsilon\int_0^t  e^{s}\|z(s)-z^*\|\rmd s. $$
Thus, by Gronwall inequality, if $\|z(s)-z^*\|\le\left(\frac{\epsilon}{k}\right)^{1/\beta}$ for all $0\le s\le t$, then
$$\|z(t)-z^*\| \le \| z(0)-z^*\| e^{-(1-\epsilon)t}.$$
But this also implies that if $\|z(0)-z^*\|\le\left(\frac{\epsilon}{k}\right)^{1/\beta}$ then $\|z(t)-z^*\|\le \left(\frac{\epsilon}{k}\right)^{1/\beta}$ for all $t\ge 0$. Hence, for all $t\ge 0$ and any small $\epsilon>0$ and $z(0)$ such that $\|z(0)-z^*\|\le \left(\frac{\epsilon}{k}\right)^{1/\beta}$, we have 
\begin{equation*}\|z(t)-z^*\|\le  e^{-(1-\epsilon) t}\|z(0)-z^*\|.\end{equation*}
\end{proof}

\begin{lemma}\label{lem:CVspeed}
Let $z^*=e_j$ be a stable equilibrium, with $j\in V$. Then, a.s. on the event  $\{Z_t\to z^*\}$, for all $\epsilon>0$,
$$\sum_{i\ne j} L(i,t)=o(t^{\epsilon}).$$
\end{lemma} 

\begin{proof}
Let us fix $\epsilon>0$ and let $\delta_{\epsilon}$ be the constant defined in Lemma \ref{cvrate}. Note that on the event $\Gamma(z^*):=\{Z_t\to z^*\}$, there exists $T_{\epsilon}>0$ such that $\tilde{Z}_t\in B_{\delta_{\epsilon}}$ for all $t\ge T_{\epsilon}$.
Combining the results in Theorem \ref{cvthrm} with Lemma \ref{cvrate} and using Lemma 8.7 in \cite{Benaim99}, we have  a.s. on $\Gamma(z^*),$
$$\limsup_{t\to\infty} \frac{1}{t}\log\|\tilde{Z}_{t}-z^*\|  \le -1+\epsilon$$
for arbitrary $\epsilon>0$. This implies that  a.s. on $\Gamma(z^*),$ that $\| {Z}_{t}- z^*\|=o(t^{-(1-\epsilon)})$. And the lemma easily follows.
\end{proof}

\begin{lemma}\label{boundlm} 
Let $j\in V$, $\epsilon\in (0,1-1/\alpha)$ and $C$ a finite constant. Set $$A_{j,C,\epsilon}:=\left\{\sum_{i\ne j} L(i,t)\le C t^\epsilon,\;\forall t\ge 1\right\}.$$ 
Then
$\E[\sum_{i\neq j} L(i,\infty) 1_{A_{j,C,\epsilon}}] < \infty$.
\end{lemma}
\begin{proof}

For each $n\ge 1$, set $\tau_n:=\inf\{t\ge 1:\, L(j,t)= n\}$ and $\gamma_n=\sum_{i\in V\setminus\{j\}} L(i,\tau_n)$.  Set also $\tau:=\inf\{t\ge 1:\; \sum_{i\ne j} L(i,t)> C t^\epsilon\}$, $\tau'_n=\tau_n\wedge \tau$ and $\gamma'_n=\sum_{i\in V\setminus\{j\}} L(i,\tau'_n)$. Note that $A_{j,C,\epsilon}=\{\tau=\infty\}$ and on $A_{j,C,\epsilon}$, $\tau_n=\tau'_n<\infty$  and $\gamma_n=\gamma'_n$ for all $n\ge 1$.

During the time interval $[\tau'_n,\tau'_{n+1}]$, the jumping rate to $j$ is larger than $\rho_0= n^{\alpha}$ and the jumping rate from $j$ is smaller than $\rho_1=(C (n+1)^{\epsilon})^{\alpha}$.
This implies that on the time interval $[\tau'_n,\tau'_{n+1}]$, the number of jumps from $j$ to $V\setminus\{j\}$ is stochastically dominated by the number of jumps of a Poisson process with intensity $\rho_1$. Since the time spent at $j$ during $[\tau'_n,\tau'_{n+1}]$ is $L(j,\tau'_{n+1})-L(j,\tau'_n)\le 1$, the number of jumps from $j$ is stochastically dominated by a random variable $N\sim \text{Poisson}(\rho_1)$. Therefore, $\gamma'_{n+1}-\gamma'_n$, the time spent at $V\setminus\{j\}$ during $[\tau'_n,\tau'_{n+1}]$, is stochastically dominated by 
$T:=\sum_{i=1}^N \xi_{i}$, where $\xi_i,i=1,2,...,N$ are independent and exponentially distributed random variables with mean value $1/\rho_0.$
Therefore,
$$\E[\gamma'_{n+1}-\gamma'_n]\le \frac{\rho_1}{\rho_0}=\frac{C^\alpha(n+1)^{\alpha\epsilon}}{n^\alpha}=O\left(\frac{1}{n^{\alpha(1-\epsilon)}}\right).$$
Since $\lim_{n\to\infty}\gamma'_n=\sum_{i\ne j}L(i,\tau)$, this proves that
$\E\left[\sum_{i\ne j}L(i,\tau)\right]<\infty.$
This proves the lemma since $\sum_{i\neq j} L(i,\infty) 1_{A_{j,C,\epsilon}} \le \sum_{i\ne j}L(i,\tau)$.
\end{proof}

\begin{theorem}\label{thm:localization}
Let $z^*=e_j\in \mathcal{S}$ be a stable equilibrium, with $j\in \{1,2,...,d\}$. Then, a.s. on the event  $\{Z_t\to z^*\}$, 
$$L(j,\infty)=\infty \quad\text{ and }\quad \sum_{i\ne j} L(i,\infty)<\infty.$$
\end{theorem} 

\begin{proof}
Lemma \ref{lem:CVspeed} implies that for $\epsilon\in (0,1-\frac{1}{\alpha})$, the event $\{Z_t\to z^*\}$ coincides a.s. with $\cup_{C} A_{j,C,\epsilon}$. Lemma \ref{boundlm} states that for all $C>0$, a.s. on $A_{j,C,\epsilon}$, $\sum_{i\neq j} L(i,\infty)<\infty$. Therefore, we have that a.s. on $\{Z_t\to z^*\}$, $\sum_{i\neq j} L(i,\infty)<\infty$.
\end{proof}

We will show in the next section that if $z^*$ is an unstable equilibrium, then $\P(Z_t\to z^*)=0$ and therefore
this will finish the proof of Theorem \ref{thm:mainresult}.
%\sout{then $\P(Z_t\to z^*)=0$ and thus obtain our following main result:}}
%{\color{blue}\begin{theorem}
%\sout{Assume that $X_t$ is a strongly VRJP in a complete graph with weight function $w(t)=t^{\alpha}$, $\alpha>1$. Then there almost surely exists a vertex $j$ such that its local time tends to infinite while the local times at the remaining vertices remain bounded.} 
%\end{theorem}}

\section{Non convergence to unstable equilibria}\label{sec:nonCV}

In this section, we prove a general non convergence theorem for a class of finite variation c\`adl\`ag processes. 
The proof of this theorem follows ideas from the proof of a theorem of Brandi\`ere and Duflo (see \cite{Brandiere96} or \cite{Duflo1996}), but using a new idea presented in Section \ref{sec:dirattract}, where sufficient conditions are given for an asymptotic pseudo-trajectory $Z$ of a dynamical system to be attracted exponentially fast towards the unstable manifold of an equilibrium $z^*$ on the event $Z_t$ converges towards $z^*$.
Then, in Section \ref{sec:dirinst}, we prove a non convergence theorem towards an unstable equilibrium that has no stable direction. The proof essentially follows \cite{Brandiere96} and \cite{Duflo1996}. 
We also point out in Remark \ref{rk:inaccuracy} several inaccuracies in their proof. 

The results proved in Sections \ref{sec:dirattract} and \ref{sec:dirinst} are then applied in Section \ref{sec:nonCV_VRJP} to strongly VRJP, showing in particular that the occupation measure process does not converge towards unstable equilibria with probability 1.
\subsection{Attraction towards the unstable manifold}\label{sec:dirattract}

In this section, we fix $m\in\{1,2,\dots d\}$, a point $z\in\mathbb{R}^d$ will be written as $z=(x,y)$ where $x\in \mathbb R^m$ and $y\in \mathbb R^{d-m}$.
Let $\Pi:\mathbb{R}^d\to\mathbb{R}^m$ be defined by $\Pi(x,y)=x$ (since $\Pi$ is linear, we will often write $\Pi z$ instead of $\Pi(z)$).

We let $F:\mathbb{R}^d\to\mathbb{R}^d$ be a $C^1$ Lipschitz vector field. 

Let us consider a finite variation c\`adl\`ag process $Z=(X,Y)$ in $\mathbb{R}^d$, adapted to a filtration $(\mathcal{F}_t)_{t\ge 0}$, satisfying the following equation 
\begin{align*}Z_t-Z_s=\int_s^t F(Z_u)\rmd u+\int_s^t {\Psi}_u\rmd u + M_t-M_s\end{align*}
where $M_t$ is a finite variation c\`adl\`ag martingale w.r.t $(\mathcal{F}_t)$ and $\Psi_t$ is a $(\mathcal{F}_t)$-adapted process.

Let $z^*=(x^*,y^*)$ be an equilibrium of $F$, i.e. $F(z^*)=0$. In the following, $\Gamma$ denotes the event $\{\lim_{t\to\infty} Z_t = z^*\}$.
 
\begin{hypothesis}\label{hyp:gpt}
There is $\gamma>0$ such that for all $T>0$, there exists a finite constant $C(T)$, such that for all $t>0$, $$\E\left[\sup_{0\le h\le T} \left\|\int_{t}^{t+h} (\Psi_u\rmd u +\rmd M_u)\right\|^2\right]\le C(T) e^{-2\gamma t}.$$
\end{hypothesis}
\begin{remark}
Using Doob's inequality, Hypothesis \ref{hyp:gpt} is satisfied as soon as there is a constant $C>0$ such that for all $t>0$,  $\|\Psi_t\|\le C e^{-\gamma t}$ and for all $t>s>0$ and all $1\le i\le d$, $\langle M^i\rangle_t-\langle M^i\rangle_s\le Ce^{-2\gamma s}$.
\end{remark}

\begin{lemma}\label{lem:gpt}
If Hypothesis \ref{hyp:gpt} holds, then
$Z$ is a $\gamma$-pseudotrajectory of $\Phi$, the flow generated by $F$, i.e. a.s. for all $T>0$
$$\limsup_{t\to\infty}\frac{1}{t}\log\left(\sup_{0\le h\le T} \|Z_{t+h}-\Phi_h(Z_t)\|\right)\le -\gamma.$$ 
\end{lemma}
\begin{proof}
Follow the proof of Proposition 8.3 in \cite{Benaim99}.
\end{proof}

\begin{hypothesis}
\label{hyp:dirattract}
There are $\mu>0$ and $\mathcal{N}=\mathcal{N}_1\times \mathcal{N}_2$ a compact convex neighbourhood of $z^*$ (with $\mathcal{N}_1$ and $\mathcal{N}_2$ respectively neighbourhoods of $x^*\in\mathbb{R}^m$ and of $y^*\in \mathbb{R}^{d-m}$) such that 
$K:=\{z=(x,y)\in \overline{\mathcal{N}}:y=y^*\}$ attracts exponentially $\overline{\mathcal{N}}$ at rate $-\mu$ (i.e. there is a constant $C$ such that $d(\Phi_t(z),K)\le C e^{-\mu t}$ for all $t>0$).
\end{hypothesis}

\begin{lemma}\label{lem:dirattract}
If Hypotheses \ref{hyp:gpt} and \ref{hyp:dirattract} hold, then,
setting $\beta_0:=\gamma\wedge \mu$, for all $\beta\in (0,\beta_0)$, on the event $\Gamma$,
\begin{equation}
\|Y_t-y^*\| = O(e^{-\beta t}).
\end{equation}
\end{lemma}
\begin{proof}
This is a consequence of Lemma 8.7 in \cite{Benaim99}.
\end{proof}

\begin{hypothesis} \label{hyp:alpha-holder}
Suppose there are $\alpha>1$ and $C>0$ such that for all $1\le i\le m$ and all $(x,y)\in \mathcal{N}$,
$$|F_i(x,y)-F_i(x,y^*)|\le C \|y-y^*\|^\alpha.$$
\end{hypothesis}

Set $G:\mathbb{R}^m\to\mathbb{R}^m$ be the $C^1$ vector field defined by $G_i(x)=F_i(x,y^*)$, for $1\le i\le m$ and $x\in\mathbb{R}^m$.
For $p>0$, denote $$\Gamma_p:=\Gamma\cap \{\forall t\ge p:\; Z_t\in \mathcal{N}\}.$$ 
For $1\le i\le m$, set $$\tilde{\Psi}_i(t)=\Psi_i(t) + F_i(X_t,Y_t) - F_i(X_t,y^*).$$
\begin{lemma}\label{lem:reducx}
Under Hypotheses \ref{hyp:gpt}, \ref{hyp:dirattract} and \ref{hyp:alpha-holder},  on $\Gamma_p$, it holds that, as $t\to\infty$,  $$\tilde{\Psi}_t=\Pi \Psi_t + O(e^{-\alpha\beta t})$$ for all $\beta\in (0,\beta_0)$ and  that for all $p<s<t$,
$$X_t-X_s=\int_s^t G(X_u)\rmd u+\int_s^t \tilde{\Psi}_u\rmd u + \Pi M_t-\Pi M_s.$$
\end{lemma}
\begin{proof} This lemma is a straightforward consequence of Lemma \ref{lem:dirattract}.
\end{proof}

\subsection{Avoiding repulsive traps}\label{sec:dirinst}
In applications, this subsection will be used for the process $X$ defined in Lemma \ref{lem:reducx}.

In this subsection, we let $F:\mathbb{R}^d\to\mathbb{R}^d$ be a $C^1$ Lipschitz vector field and we consider a finite variation c\`adl\`ag process $Z$ in $\mathbb{R}^d$, adapted to a filtration $(\mathcal{F}_t)_{t\ge 0}$, satisfying the following equation 
$$Z_t-Z_s=\int_s^t F(Z_u)\rmd u+\int_s^t {\Psi}_u\rmd u + M_t-M_s$$
where $M_t$ is a finite variation c\`adl\`ag martingale w.r.t $(\mathcal{F}_t)$ and $\Psi_t=r_t+R_t$, with $r$ and $R$ two $(\mathcal{F}_t)$-adapted processes.

Let $z^*\in\mathbb{R}^d$ and $\Gamma$ an event on which $\lim_{t\to\infty} Z_t = z^*$.
Let $\mathcal{N}$ be a convex neighbourhood of $z^*$.
For $p>0$, set $$\Gamma_p:=\Gamma\cap \{\forall t\ge p:\; Z_t\in \mathcal{N}\}.$$ Then, $\Gamma=\cup_{p>0}\Gamma_p$.

We will suppose that
\begin{hypothesis}\label{hyp:z*}
$z^*$ is a repulsive equilibrium, i.e. $F(z^*)=0$ and all eigenvalues of $DF(z^*)$ have a positive real part. Moreover $DF(z^*)=\lambda I$, with $\lambda>0$ and $I$ the identity $d\times d$ matrix. 
\end{hypothesis}

%\begin{remark}
%{\color{blue} I think this remark is incorrect and should be deleted. \sout{
%The second assumption is not necessary (see \cite{Brandiere96}, sections I.3 and I.4.) and we can suppose more generally that $z^*$ is a repulsive equilibrium, and by taking for $\lambda$ a positive constant lower than the real part of any eigenvalue of $DF(z^*)$. Our proof is however  is a little easier to write with Hypothesis \ref{hyp:z*}.}}
%\end{remark}

For all $z\in\mathbb{R}^d$,
\begin{align*}
F(z)&=F(z^*)+\int_0^1 DF(z^*+u(z-z^*)).(z-z^*) \rmd u\\
%&= \lambda z + \int_0^1 (DF(uz)-DF(z^*))\rmd u.(z-z^*) \\
&= \lambda (z-z^*) + J(z).(z-z^*)
\end{align*}
where we have set $$J(z)=\int_0^1 (DF(z^*+u(z-z^*))-DF(z^*))\rmd u.$$ 
Then, for all $t\ge s$,
\begin{equation}\label{eq:zl}
Z_t-Z_s=\int_s^t \lambda Z_u \rmd u+\int_s^t \left[\Psi_u+J(Z_u).(Z_u-z^*)\right]\rmd u + M_t-M_s.
\end{equation} 

%We will also suppose that
%\begin{hypothesis}\label{hyp:DFlip}
%$DF$ is Lipschitz on $\mathcal{N}$. 
%\end{hypothesis}
%\textsc{It seems that this assumption is not necessary!!}
%
%Hypothesis \ref{hyp:DFlip} ensures that, for all $z\in \mathcal{N}$, $\|H(z)\|\le C \|z-z^*\|$, with $C$ the Lipschitz constant of $DF$ restricted to $\mathcal{N}$. 
%Thus on $\Gamma_p$, we have that for all $t>p$, $\|H(Z_t)\|\le C \|z_t-z^*\|$.

%We choose the neighborhood $\mathcal{N}$ sufficiently small such that for all $z\in \mathcal{N}$
%\begin{equation}
%\|\Delta(z)\|\le 
%\end{equation}

Let us fix $p>0$.
Note that \eqref{eq:zl} implies that, for all $t\ge p$,
\begin{equation}\label{eq:zl2}
Z_t=e^{\lambda t} \left(e^{-\lambda p} Z_p +\int_p^t\bar{\Psi}_sds+ \bar M_t-\bar M_p\right)
\end{equation}
where $\bar M_t= \int_0^t e^{-\lambda s} \rmd M_s$ and $\bar{\Psi}_t=\bar{r}_t+\bar{R}_t$, with
$$\bar{r}_t:= e^{-\lambda t}r_t \quad \text{  and  } \quad \bar{R}_t:=e^{-\lambda t}[R_t+J(Z_t).(Z_t-z^*)] .$$
We assume that the following hypothesis is fulfilled:
%\begin{hypothesis} \label{hyp:RM} 
%\begin{enumerate}[(i)] 
%\item There is a function $\alpha: [0,\infty)\to (0,\infty)$ converging towards $0$ as $t\to\infty$ and positive constants $c_-<c_+$ such that almost surely
%\begin{align}\label{eq:liminf}
% \liminf_{t\to\infty}\frac{1}{\alpha^2(t)}\sum_{i=1}^d\E\big([M^i]_\infty-[M^i]_t\big|\mathcal{F}_t\big)& \ge c_-\P(\Gamma),\\
% \label{eq:limsup}
%\limsup_{t\to\infty}\frac{1}{\alpha^2(t)}\sum_{i=1}^d\E\big([M^i]_\infty-[M^i]_t\big|\mathcal{F}_t\big) & \le c_+.
%\end{align}
%\item As $t\to\infty$,
%$$\E\left[1_\Gamma\int_t^\infty \|{R}_s\| ds\right]=o(\alpha(t)).$$ 
%\end{enumerate}
%\end{hypothesis}

\begin{hypothesis} \label{hyp:RM}
There is a random variable $K$ finite on $\Gamma$ and
there is a continuous function $a:[0,\infty)\to (0,\infty)$ such that $\int_0^\infty a(s)\rmd s <\infty$, 
 $\alpha^2(t) :=\int_t^{\infty}a(s)\rmd s=O\big( \int_t^\infty e^{-2\lambda (s-t)} a(s) \rmd s\big)$ as $t\to\infty$
and such that the following items \textit{(i)} and \textit{(ii)} hold.
\begin{enumerate}[(i)] 
\item For each $i$, $\langle M^i\rangle_t = \int_0^t \Lambda^i_s \rmd s$, with $\Lambda^i$ a positive $(\mathcal{F}_t)$-adapted process. Setting $\Lambda=\sum_i \Lambda^i$, we have that a.s. on $\Gamma$, for all $t>0$,
\begin{align}\label{eq:THM1}
 K^{-1}a(t) \le \Lambda(t)\le Ka (t),
\end{align}
\begin{align}\label{eq:THM2}
\sum_{i=1}^d|\Delta M^i_t|\le K \alpha(t),
\end{align}
\begin{align} \label{eq:THM4}
\int_0^\infty\frac{\|r_s\|^2}{a(s)} \rmd s \le K.
\end{align}
\item As $t\to\infty$,
\begin{align}\label{eq:THM3}\E\left[1_\Gamma \left(\int_t^\infty \|{R}_s\| \rmd s\right)^2\right]=o\big(\alpha^2(t)\big).
\end{align} 
\end{enumerate}
\end{hypothesis}

For $p>0$, define $$G_p=\Gamma_p\cap\{\sup_{t\ge p}\|J(Z_t)\|\le \frac{\lambda}{2}\}\cap\{\sup_{t\ge p}\|Z_t\|\le 1\}.$$

%----------------------

%Assumption \textit{(i)} of Hypothesis \ref{hyp:RM} is the required assumption for the martingale $M$.
%
%We have (using Doob's inequality and hypothesis \ref{hyp:RM}), for $t>p$,
%\begin{align*}
%\E[1_{G_p} \|S_t\|]
%&\le \; \E\left[1_{G_p}\int_t^\infty \|{R}_s\| ds\right]
%+ 2 \E\big[ [M]_\infty-[M]_t\big]^{\frac{1}{2}}= O(\alpha(t)).
%\end{align*}
%It follows that $$\E[1_{G_p} \int_t^\infty e^{-\lambda (s-p)}\|S_s\|ds]=o(\alpha(t)).$$ 

%------------------

\begin{lemma} \label{lem:R}
For all $p>0$, as $t\to\infty$,
\begin{align}\label{eq:R}
\E\left[1_{G_p} \int_t^\infty \|\bar{R}_s\|\rmd s\right]=o(e^{-\lambda t}\alpha(t)).
\end{align} 
\end{lemma}
\begin{proof} Fix $p>0$.
Since Hypothesis \ref{hyp:RM}-(ii) holds, to prove the lemma it suffices to prove that as $t\to\infty$,
$$\E\left[1_{G_p}\int_t^\infty e^{-\lambda s}\|J(Z_s).(Z_s-z^*)\| \rmd s\right]= o(e^{-\lambda t}\alpha(t)).$$ 
To simplify the notation, we suppose $z^*=0$. For $s<t$, (using the convention: $\frac{z}{\|z\|}=0$ if $z=0$)
\begin{align*}
\|Z_t\|-\|Z_s\|
=& \;\lambda\int_s^t \|Z_u\| \rmd u + \int_s^t \left\langle \frac{Z_u}{\|Z_u\|},J(Z_u)Z_u\right\rangle \rmd u\\
& + \int_s^t \left\langle \frac{Z_{u-}}{\|Z_{u-}\|}, \rmd M_u\right\rangle +  \int_s^t \left\langle \frac{Z_u}{\|Z_u\|}, \Psi_u\right\rangle \rmd u\\
&+ \sum_{s<u\le t} 1_{\{Z_{u-}\ne 0\}}\left(\Delta \|Z_u\| - \left\langle \frac{Z_{u-}}{\|Z_{u-}\|},\Delta Z_u\right\rangle \right).
\end{align*}
Using the inequality $\|z+\delta\|-\|z\|\ge \langle \frac{z}{\|z\|},\delta \rangle$, we have for all $u>p$,
$$\Delta \|Z_u\| -\left\langle \frac{Z_{u-}}{\|Z_{u-}\|},\Delta Z_u\right\rangle\ge 0.$$
Furthermore, using Cauchy-Schwarz inequality, on the event $G_p$,
 $$\left\langle \frac{Z_u}{\|Z_u\|},J(Z_u)Z_u\right\rangle \ge - \|J(Z_u)Z_u\|\ge -\sup_{t\ge p}\|J(Z_t)\|.\|Z_u\|\ge-\frac{\lambda}{2}\|Z_u\|$$
for all $u> p$. From the above it follows that on the event $G_p$,
\begin{align*}
\|Z_t\|-\|Z_s\|
\ge& \;\frac{\lambda}{2}\int_s^t \|Z_u\| \rmd u  + \int_s^t \left\langle \frac{Z_{u-}}{\|Z_{u-}\|}, \rmd M_u\right\rangle + \int_s^t \left\langle \frac{Z_u}{\|Z_u\|},\Psi_u\right\rangle \rmd u
\end{align*}
for all $t>s>p$. As a consequence, using Doob's inequality and Hypothesis \ref{hyp:RM}, we obtain that
\begin{eqnarray*}
\frac{\lambda}{2}\E\left[1_{G_p}\left(\int_t^\infty  \|Z_s\|\rmd s\right)^2\right]^\frac{1}{2}
&\le&\; \E\left[1_{G_p}\sup_{T>t}\left|\int_t^T \left\langle \frac{Z_{u-}}{\|Z_{u-}\|}, \rmd M_u\right\rangle\right|^2\right]^\frac{1}{2}\\
& &+ \; \alpha(t) \E\left[1_{G_p}\int_t^\infty \frac{\|r_u\|^2}{a(u)} \rmd u\right]^\frac{1}{2}\\
& &+ \; \E\left[1_{G_p}\left(\int_t^\infty \|R_u\| \rmd u\right)^2\right]^\frac{1}{2}\\
&= &\; O(\alpha(t)).
\end{eqnarray*}

Using Cauchy-Schwarz inequality, we have
\begin{align*}
\E\left[1_{G_p}\int_t^\infty  e^{-\lambda s}\|J(Z_s)Z_s\|\rmd s\right]
\le & \; e^{-\lambda t}\E\left[1_{G_p}\sup_{s\ge t}  \|J(Z_s)\|^2\right]^{\frac12}
\E\left[1_{G_p}\left(\int_t^\infty  \|Z_s\|\rmd s\right)^2\right]^{\frac12}.
\end{align*}
Note that on $G_p$, $\sup_{s\ge t}  \|J(Z_s)\|\le \lambda/2$ and $\lim_{t\to\infty}\sup_{s\ge t}  \|J(Z_s)\|=0$ almost surely. Therefore, we conclude that 
$\E[1_{G_p}\int_t^\infty  e^{-\lambda s} \|J(Z_s)Z_s\|\rmd s]=o(e^{-\lambda t} \alpha(t))$ as $t\to \infty$.
\end{proof}

Hypothesis \ref{hyp:RM} ensures in particular that a.s. on $G_p$, $\int_p^\infty \bar{\Psi}_s\rmd s$ and $\bar M_{\infty}$ are well defined and almost surely finite. Let $L$ be a random variable such that $$L=\int_p^\infty \bar{\Psi}_s\rmd s+\bar M_{\infty}-\bar M_{p}\ \ \text{ on  } G_p.$$
Letting $t\to\infty$ in \eqref{eq:zl2}, $\lambda$ being positive, we have $L=-e^{-\lambda p} Z_p\ \text{\  a.s. on }G_p.$
%Set $$S_t=\int_0^t \bar{R}_sds+ \bar{M}_t.$$
We now apply Theorem \ref{THM:THMA} to the martingale $\bar M_t$ and to the adapted process $\bar \Psi_t$. We have $\langle \bar{M}^i\rangle_t=\int_0^t \bar{\Lambda}^i_s \rmd s,$ with $\bar{\Lambda}^i_s=e^{-2\lambda s}\Lambda^i_s$. 
We also have $|\Delta \bar{M}_t|=e^{-\lambda t}|\Delta M_t|$.  
Hypothesis \ref{hyp:RM}-(i) implies that \eqref{eq:THMA1}, \eqref{eq:THMA2} and \eqref{eq:THMA4} are satisfied with the function $\bar{a}(t)=e^{-2\lambda t}a(t)$. Finally, \eqref{eq:THMA3} follows from Lemma \ref{lem:R}. Therefore,  we obtain that
 $$\P(G_p)=\P(G_p\cap\{L=-e^{-\lambda p} Z_p\}]=0.$$  

Since $\P(\Gamma)=\lim_{p\to\infty}\P(G_p)=0$, we have proved the following theorem:
\begin{theorem}\label{THM:nonCV}
Under Hypotheses \ref{hyp:z*} and \ref{hyp:RM}, we have $\P(\Gamma)=0$.
\end{theorem}

\subsection{Application to strongly VRJP on complete graphs}
\label{sec:nonCV_VRJP}

%To simplify a little, we will suppose here that $\ell_0=1$.
Recall from Section \ref{sec:Dyn} that the empirical occupation measure process $(Z_t)_{t\ge0}$ satisfies the following equation
\begin{align}\label{vrjp}
Z_t-Z_s &=\;
\int_s^t \frac{1}{u+\ell_0} F(Z_u) \rmd u + \frac{I[X_s]}{(s+\ell_0)w_s}- \frac{I[X_t]}{(t+\ell_0)w_t}\\ \nonumber
&+ \int_s^t \Psi_u \rmd u + \int_s^t \frac{\rmd M_u}{(u+\ell_0)w_u},
\end{align}
where $$\Psi_t=I[X_t]\frac{\rmd}{\rmd t}\left(\frac{1}{(t+\ell_0)w_t}\right)\quad \text{ and } \quad M_t=I[X_t]-\int_0^t A_s[X_s] \rmd s.$$
Recall that $\langle M^j\rangle_t=\int_0^t \Lambda^j_s \rmd s$, where $\Lambda^j$ is defined in \eqref{eq:defLambdaj}.

For $t\ge t_0:=\log(\ell_0)$, let 
\begin{equation}
 \label{def:Zhat}
 \widehat{Z}_t=Z_{e^t-\ell_0}+\frac{I[X_{e^t-\ell_0}]}{e^tw_{e^t-\ell_0}}.
\end{equation}
Equation (\ref{vrjp}) is thus equivalent to
\begin{align}\label{mvrjp}\widehat{Z}_t-\widehat{Z}_s=\int_s^t F(\widehat{Z}_u)\rmd u+\int_s^t \widehat{\Psi}_u\rmd u + \widehat{M}_t-\widehat{M}_s,
\end{align}
where we have set
\begin{align*}
\widehat{\Psi}_t=e^t \Psi_{e^t-\ell_0} +F(Z_{e^t-\ell_0})-F(\widehat{Z}_t) \qquad \hbox{and} \qquad \widehat{M}_t=\int_0^{e^t-\ell_0}\frac{\rmd M_s}{(s+\ell_0)w_s},
\end{align*}
which are respectively an adapted process and a martingale w.r.t the filtration $(\widehat{\mathcal{F}}_t)_{t\ge t_0}:=(\mathcal{F}_{e^t-\ell_0})_{t\ge t_0}$.
Note that $\langle \widehat{M}^j\rangle_t-\langle \widehat{M}^j\rangle_{t_0}=\int_{t_0}^t \widehat{\Lambda}^j_s \rmd s$, with $\widehat{\Lambda}^j_s=\frac{\Lambda^j_{e^s-\ell_0}}{e^sw^2_{e^s-\ell_0}}.$

In this subsection, we will apply the results of Subsection \ref{sec:dirattract} and Subsection \ref{sec:dirinst} to the process $(\widehat{Z}_t)_{t\ge0}$ and thus show that $P[Z_{t}\to z^*]=P[\widehat{Z}_{t}\to z^*]=0$ for each unstable equilibrium $z^*$.

\begin{lemma}\label{lem:togpt}
 There exists a positive constant $K$ such that for all $t>t_0$, a.s.
\begin{align*}
&\|\widehat{\Psi}_t\|\le K e^{-(\alpha+1)t},\quad \quad  \widehat{\Lambda}^j_t \le K e^{-(\alpha+1)t} \hbox{ and } \quad |\Delta \widehat{M}_t^j|\le K e^{-(\alpha+1)t}.
%&\E\big[[\widehat M^j]_t-[\widehat M^j]_s\big|\widehat{\mathcal{F}}_s\big]\le c_+ e^{-(\alpha+1)s} \big(1- e^{-(\alpha+1)(t-s)}\big), \hbox{ for all } t>s.
\end{align*}
\end{lemma}
\begin{proof} Let us first recall  that $w_t\ge k(t+\ell_0)^\alpha$ for some constant $k$. 
Using that $F$ is Lipschitz, we easily obtain the first inequality. To obtain the second inequality, observe that for each $j$, $\Lambda^j_t\le w_t$. Thus for all $t>t_0$,
$$\widehat{\Lambda}^j_t\le \frac{1}{e^tw_{e^t-\ell_0}}\le k^{-1} e^{-(\alpha+1)t}.$$
Finally,
$$|\Delta \widehat{M}_t^j|=\frac{|\Delta I[X_{e^t-\ell_0}]|}{e^tw_{e^t-\ell_0}}\le \frac{1}{e^tw_{e^t-\ell_0}}
\le k^{-1}e^{-(\alpha+1)t}.$$ 

\end{proof}
\begin{theorem}\label{thm:nonCV_VRJP}
Assume that $z^*$ is an unstable equilibrium of the vector field $F$ defined by \eqref{vecF}. Then 
$\P[Z_t\to z^*]=0$.
\end{theorem}

\begin{proof}
Note first that Lemma \ref{lem:togpt} implies that Hypothesis \ref{hyp:gpt} holds with $\gamma=\frac{\alpha+1}{2}$.

%Let $z^*$ be an unstable equilibrium, with $z^*_i=\frac{1}{m}$ if $i\in\{1,\dots,m\}$ and $z^*_i=0$ if $i\in\{m+1,\dots,d\}$, with $m\in\{2,3,\dots, d\}$ (up to a permutation of indices, this describes the set of all unstable equilibriums). A point $z\in \mathbb{R}^d$ will be written in the form $z=(x,y)$, with $x=(x_1,\dots x_m)\in\mathbb{R}^m$ and $y=(y_{m+1},\dots,y_d)\in \mathbb{R}^{d-m}$.

Let $z^*=(x^*,y^*)$ be an unstable equilibrium, where $y^*=0\in\mathbb{R}^{d-m}$ and $x^*=\left(\frac1m,\frac1m,\dots,\frac1m\right)\in \mathbb{R}^m$, with $m\in\{2,3,\dots, d\}$ (up to a permutation of indices, this describes the set of all unstable equilibria).

Note also that there is a compact convex neighbourhood $\mathcal{N}=\mathcal{N}_1\times \mathcal{N}_2$ of $z^*$ and a positive constant $h$ such that for all $z\in\mathcal{N}$, $H(z)=\sum_i z_i^\alpha \ge h$. 
Setting $C(z)=\frac{1}{H(z)}$, we have that for all $i\in\{1,2,\dots,d-m\}$,
$$ F_{m+i}(x,y)=-y_{i}(1+ C(z)y_{i}^{\alpha-1}).$$
Since $\alpha>1$, it can easily be shown that Hypothesis \ref{hyp:dirattract} holds for all $\mu\in (0,1)$. 
Hypothesis \ref{hyp:alpha-holder} also holds (with the same constant $\alpha$). 

Therefore, Lemma \ref{lem:reducx} can be applied to the process $(\widehat{Z}_t)_{t\ge t_0}$ defined by (\ref{def:Zhat}). Set $\widehat{X}_t:=\Pi\widehat{Z}_t$ and let ${G}:\mathbb{R}^m \to\mathbb{R}^m$ be the vector field defined by $G_i(x)=F_i(x,0)$. Then for all $s<t$,
$$\widehat{X}_t-\widehat{X}_s=\int_s^t G(\widehat{X}_u)\rmd u+\int_s^t 
\hat{r}_u\rmd u + \Pi\widehat{M}_t-\Pi\widehat{M}_s,$$
with $\hat{r}_t=\Pi\widehat{\Psi}_t+O(e^{-\alpha\beta t})$ on $\Gamma$, for all $\beta<\gamma\wedge \mu$. Note that since $\mu$ can be taken as close as we want to $1$ and since $\gamma=\frac{\alpha+1}{2}>1$, $\beta$ can be also taken as close as we want to $1$.

We now apply the result of Section \ref{sec:dirinst}, with $Z$, $F$, $M$, $r$ and $R$ respectively replaced by $\hat{X}$, $G$, $\Pi\hat{M}$, $\hat{r}$ and $0$.
The vector field $G$ satisfies Hypothesis \ref{hyp:z*} with $\lambda=\alpha-1$. 
%Hypothesis \ref{hyp:DFlip} is also satisfied (taking $\mathcal{N}$ sufficiently small).
Let us now check Hypothesis \ref{hyp:RM} with $a(t)=e^{-(\alpha+1)t}$. 
Choosing $\beta\in (\frac{\alpha+1}{2\alpha},1)$, we have that $\hat{r}$ satisfies \eqref{eq:THM4}. 

Set $\widehat{\Lambda}=\sum_{j=1}^m \widehat{\Lambda}^j$. It remains to verify the inequality (\ref{eq:THM1}) for $\widehat{\Lambda}$. Lemma \ref{lem:togpt} shows that for all $t>0$, $$\widehat{\Lambda}_t\le \frac{m}{k}e^{-(\alpha+1)t} = C_+ e^{-(\alpha+1)t}.$$

Fix $\epsilon\in (0,1)$ and choose the neighbourhood $\mathcal{N}$ sufficiently small such that for all $z\in \mathcal{N}$ and $i\in\{1,\dots,m\}$, $m\pi_i(z)\in (1-\epsilon,1+\epsilon)$. Therefore, if $Z_t\in\mathcal{N}$, we have that for $i\in\{1,\dots,m\}$, $w^{(j)}_t=w_t\pi_j(Z_t)\ge \frac{k(1-\epsilon)}{m} (t+\ell_0)^\alpha$. 
Therefore, since $m\ge 2$, if $Z_t\in\mathcal{N}$, we have that for all $1\le i\le m$
$$\Lambda^i_t\ge 1_{\{X_t=i\}} \sum_{j\neq i, 1\le j\le m} w^{(j)}_t + 1_{\{X_t\neq i\}} w^{(i)}_t\ge \min_{1\le j \le m}w_u^{(j)}\ge \frac{k(1-\epsilon)}{m}(u+\ell_0)^\alpha.$$ 
Since $w_t\le d(t+\ell_0)^{\alpha}$, we have that if $Z_t\in\mathcal{N}$,
$$\widehat{\Lambda}_t\ge \frac{k(1-\epsilon)e^{\alpha t}}{e^td^2e^{2\alpha t}}=C_- e^{-(\alpha+1)t}.$$
This proves that Hypothesis \ref{hyp:RM} is satisfied.

As a conclusion Theorem \ref{THM:nonCV} can be applied, and this proves that $\P[Z_t\to z^*]=\P[\widehat X_t\to x^*]=0$.
\end{proof}

\subsection{A theorem on martingales}

In this subsection, we prove a martingale theorem, which is a continuous time version of a theorem by Brandi\`ere and Duflo (see Theorem A in \cite{Brandiere96} or Theorem 3.IV.13 in \cite{Duflo1996}). 
\begin{theorem}\label{THM:THMA}
Let  $M$ be a finite variation c\`adl\`ag martingale in $\mathbb{R}^d$ with $M_0=0$, $r$ and $R$ be adapted processes in $\mathbb{R}^d$ with respect to a filtration $(\mathcal{F}_t)_{t\ge0}$. Set $\Psi_t=r_t+R_t$.

Let $\Gamma$ be an event and let $a:[0,\infty)\to (0,\infty)$ be a continuous function such that $\int_0^\infty a(s)\rmd s <\infty$ and set $\alpha^2(t)=\int_t^\infty a(s) \rmd s$.
Suppose that for each $i$, $\langle M^i\rangle_t = \int_0^t \Lambda^i_s \rmd s$, with $\Lambda^i$ a positive adapted c\`adl\`ag process. Set $\Lambda=\sum_i \Lambda^i$.
Suppose that there is a random variable $K$, such that a.s. on $\Gamma$, $1<K<\infty$ and for all $t>0$,
\begin{align}\label{eq:THMA1}
 K^{-1} a(t) \le \Lambda(t)\le Ka (t).
\end{align}
\begin{align}\label{eq:THMA2}
\sum_i|\Delta M^i_t|\le K \alpha(t).
\end{align}
\begin{align}\label{eq:THMA4}
\int_0^\infty \frac{\|r_s\|^2}{a(s)} \rmd s \le K
\end{align}
and as $t\to\infty$,
\begin{align}\label{eq:THMA3}
\E\left[1_\Gamma\int_t^\infty \|R_s\| \rmd s\right]= o(\alpha(t)).\end{align}

Then, a.s. on $\Gamma$, $S_t:=\int_0^t {\Psi}_s\rmd s+ M_t$ converges a.s. towards a finite random variable $L$ and for all $\mathcal{F}_p$-measurable random variable $\eta$, $p>0$, we have
$$\P[\Gamma\cap\{L=\eta\}]=0.$$
\end{theorem}
\begin{remark} \label{rk:inaccuracy} Our theorem here is a continuous-time version of Theorem A by Brandi\`ere and Duflo in \cite{Brandiere96}. Their results is widely applied to discrete stochastic approximation processes, in particular to showing the non convergence to a repulsive equilibrium. Note that there is an inaccuracy in the application of the Burkholder's inequality in their proof. Beside of this, there is also a mistake in the application of their theorem to the proof of Proposition 4 in \cite{Brandiere96} since the process $S_n$ defined in page 406 is not adapted. 
\end{remark}
\begin{proof} \ \\
\textit{Simplification of the hypotheses:}
It is enough to prove the Theorem assuming in addition that the random variable $K$ is non-random and that \eqref{eq:THMA1}, \eqref{eq:THMA2} and \eqref{eq:THMA4} are satisfied a.s. on $\Omega$.

Let us explain shortly why:  The idea is due to Lai and Wei in \cite{Lai1983} (see also \cite{Duflo1996}, p. 60-61). %Let $K=sup\{C_+,C_-^{-1},C,I\}$ on $\Gamma$ and $K=\infty$ on $\Gamma^c$. Then a.s. on $\Gamma$, $K<\infty$.
For $n\in\mathbb{N}$, let $T_n$ be the first time $t$ such $\Lambda(t)\not\in [n^{-1} a(t),na(t)]$ or $|\Delta M^i_t|>n \alpha(t)$ for some $i$ or $\int_0^t \frac{\|r_s\|^2}{a(s)} \rmd s >n$. Then $T_n$ is an increasing sequence of stopping times and a.s. on $\Gamma\cap\{K\le n\}$, $T_n=\infty$.

Possibly extending the probability space, let $N$ be a Poisson process with intensity $a(t)$.
For $n\in \mathbb{N}$, $i\in\{1,\dots,d\}$ and $t>0$, set $$\tilde{M}^{i}_t=M^i_{t\wedge T_n} + N_t-N_{t\wedge T_n}\ \text{ and } \ \tilde{r}_t=r_{t\wedge T_n}.$$ Then, $\tilde{M}$ and $\tilde{r}$ satisfy \eqref{eq:THMA1}, \eqref{eq:THMA2} and \eqref{eq:THMA4} a.s. on $\Omega$, with $K=n$, and on the event $\{T_n=\infty\}$, $\tilde{M}=M$ and $\tilde{r}=r$. Now set  $$L_n=\int_0^\infty (\tilde{r}_s+R_s) \rmd s + \tilde{M}_\infty,$$
which is well defined on $\Gamma$. Then a.s. on the event $\Gamma_n:=\Gamma\cap\{K\le n\}$, we have  $L_n=L$.

Suppose now that for all $n$, we have $\P[\Gamma_n\cap \{L_n=\eta\}]=0$, then we also have 
$\P[\Gamma\cap \{L=\eta\}]=\lim_{n\to\infty}\P[\Gamma_n\cap \{L=\eta\}]=\lim_{n\to\infty}\P[\Gamma_n\cap \{L_n=\eta\}]=0$.

\medskip
Let $\tilde\Omega$ be the event that \eqref{eq:THMA1}, \eqref{eq:THMA2} and \eqref{eq:THMA4} is satisfied with non-random positive constant $K$. 
From now on, we suppose that $K$ is non-random and that \eqref{eq:THMA1}, \eqref{eq:THMA2} and \eqref{eq:THMA4} are  satisfied a.s. on $\Omega$.

A first consequence is that, $M$, $[M^i]-\langle M^i\rangle$ and $\|M\|^2-A$, with $A=\sum_i \langle M^i\rangle$, are uniformly integrable martingales. Indeed, using Lemma VII.3.34 in \cite{Jacod2003}, p. 423, there are constant $k_1$ and $k_2$ such that
$$\E\left[{\sup}_{0\le s\le t}|M_s^i|^4\right]\le k_1\left(\sup_{0\le s\le t,\omega\in\tilde \Omega}|\Delta M_t^i(\omega)|\right)^2\left(\E\left[\langle M^i \rangle_t^2\right]\right)^{1/2}+k_2\E\left[\langle M^i \rangle_t^2\right].$$
Recall from (\ref{eq:THMA1}) and (\ref{eq:THMA2}) that $$\langle M^i\rangle_t=\int_0^t\Lambda_s^i\rmd s\le K\int_0^t a(s)\rmd s<K\int_0^{\infty} a(s)\rmd s,\quad |\Delta M^i_t|\le K\alpha(t)\le K\alpha(0)$$ for all $t\ge0$. It implies that $\E(\|M_t\|^4)$ is uniformly bounded and  $M$ is thus uniformly integrable.  

Without loss of generality, we also suppose that $p=0$ and $\eta=0$. Otherwise, one can replace $\mathcal{F}_t$, $M_t$, $r_t$ and $R_t$ by $\mathcal{F}_{t+p}$, 
$M_{t+p}-M_p$, $r_{t+p}+\beta'(t)\left(\eta-\int_0^pr_s\rmd s-M_p\right)$ and $R_{t+p}+\beta'(t)\left(\eta-\int_0^pR_s\rmd s-M_p\right)$ respectively, where $\beta:[0,\infty)\to (0,\infty)$ is some differentiable function such that $\beta(0)=1$, $\lim_{t\to\infty}\beta(t)=0$ and $\beta(t)=o(\alpha(t))$.

Set $G=\Gamma\cap \{L=0\}$. For $t\ge 0$, define $\rho_t=M_\infty-M_t,\ \tau_t=\int_t^\infty \Psi_s \rmd s\ \text{ and } T_t=\rho_t+\tau_t.$ 
Then $T_t=L-S_t$ and on $G$, $T_t=-S_t$.

Since for all $t>0$, $(\|M_{s}-M_t\|^2-(A_s-A_t),\ s\ge t)$ is a uniformly integrable martingale, we have that for all $t>0$, $\E[\|\rho_t\|^2|\mathcal{F}_t]=\E\big[ A_\infty-A_t|\mathcal{F}_t\big]=\E\big[ \int_t^{\infty}\Lambda(s)\rmd s|\mathcal{F}_t\big]$ and therefore
$$K^{-1}\alpha^2(t)\le \E[\|\rho_t\|^2|\mathcal{F}_t]\le K\alpha^2(t).$$

Using Lemma VII.3.34 in \cite{Jacod2003} to the martingale $(M_s-M_t,\ s\ge t)$, we have
\begin{align*}\E\left[|M_s^i-M_t^i|^4|\mathcal{F}_t\right]& \le k_1\left(\sup_{t\le u\le s,\omega\in\tilde \Omega}|\Delta M_u^i(\omega)|\right)^2\left(\E\left[\left(\langle M^i \rangle_s-\langle M^i \rangle_t\right)^2|\mathcal{F}_t\right]\right)^{1/2}\\
&+k_2\E\left[\left(\langle M^i \rangle_s-\langle M^i \rangle_t\right)^2|\mathcal{F}_t\right]\\
&\le k_1K^3\alpha^2(t)\int_t^s a(u)\rmd u+k_2K^2\left(\int_t^s a(u)\rmd u\right)^2.
\end{align*}

Hence, for all $t>0$, there is a constant $k$ such that $\E[\|\rho_t^4\|\mathcal{F}_t]\le k\alpha^4(t).$

Set $c_0=K^{-\frac32} k^{-\frac12}$. Since 
$\E\big[\|\rho_t\|^2|\mathcal{F}_t\big]
\le \E\big[\|\rho_t\||\mathcal{F}_t\big]^{\frac23}\E\big[\|\rho_t\|^4|\mathcal{F}_t]^\frac13,$
we have that for all $t$,
\begin{align*}
\E[\|\rho_t\||\mathcal{F}_t]& \ge c_0 \alpha(t).
\end{align*}

Let $U$ be a Borel function from $\mathbb{R}^d\setminus\{0\}$ onto the set of  $d\times d$ orthogonal matrices such that $U(a)[a/\|a\|]=e_1$ (with $e_1=(1,0,\dots,0)$).
Then on $G$,
\begin{align*}
&\|T_t\| e_1+U(S_t) T_t = 0\\
&\big\|\|\rho_t\| e_1+ U(S_t)\rho_t\big\|\le 2\|\tau_t\|.
\end{align*}

Set $G_t:=\{\P(G|\mathcal{F}_t)>\frac12\}$. Then for all $t>0$ (using in the second inequality that $S_t$ is $\mathcal{F}_t$-measurable and that $\E[\rho_t|\mathcal{F}_t]=0$)
\begin{align*}
\P(G_t)
&\le \frac{1}{c_0\alpha(t)}\big\|\E\big[ 1_{G_t}\E[\|\rho_t\|e_1|\mathcal{F}_t]\big]\big\|\\
&\le \frac{1}{c_0\alpha(t)}\big\|\E\big[ 1_{G_t}\E[\|\rho_t\|e_1 + U(S_t)\rho_t|\mathcal{F}_t]\big]\big\|\\
&\le \frac{1}{c_0\alpha(t)}\big\|\E\big[ 1_{G}\E[\|\rho_t\|e_1 + U(S_t)\rho_t|\mathcal{F}_t]\big]\big\|\\
& + \frac{1}{c_0\alpha(t)}\big\|\E\big[ (1_{G_t}-1_{G})\E[\|\rho_t\|e_1 + U(S_t)\rho_t|\mathcal{F}_t]\big]\big\|\\
&\le \frac{2}{c_0\alpha(t)}\E\big[ 1_{G}\|\tau_t\|\big] + \frac{2}{c_0\alpha(t)}\left(\E\big[ (1_{G_t}-1_{G})^2\right)^{\frac12} \left(\E[\|\rho_t\|^2]\right)^{\frac12}.
\end{align*}
Note that $$\lim_{t\to\infty}\E\big[(1_{G_t}-1_{G})^2\big]=0 \quad \text{  and  } \quad\E[\|\rho_t\|^2]\le c_+\alpha^2(t).$$ 
Thus, the second term converges to $0$. For the first term, (using Cauchy-Schwarz inequality to obtain the first term on the right hand side)
\begin{align*}
\E\big[ 1_{G}\|\tau_t\|\big]
\le & \; \E\left[ 1_{G}\int_t^\infty \|r_s\| \rmd s\right] + \E\left[ 1_{G}\int_t^\infty \|R_s\| \rmd s\right]\\
\le & \; \alpha(t) \E\left[ 1_{G}\left(\int_t^\infty \frac{\|r_s\|^2}{a(s)} \rmd s\right)^{\frac12}\right] + o(\alpha(t)) = o(\alpha(t)) 
\end{align*}
using Cauchy-Schwarz inequality, Lebesgue's Dominated Convergence Theorem and the hypotheses. We thus obtain that $\P(G)=\lim_{t\to\infty}\P(G_t)=0$.
\end{proof}

\section*{Acknowledgement}
O. Raimond's research has been conducted as part of the project Labex MME-DII (ANR11-LBX-0023-01) and of the project ANR MALIN (ANR-16-CE93-0003). T.M. Nguyen's research is partially supported by Crafoord Foundation and Thorild Dahlgren \& Folke Lann\'er Funds. The authors would like to thank the anonymous referees for their careful  reading and their valuable suggestions which improved the manuscript.
%\addcontentsline{toc}{section}{Acknowledgement}


\begin{thebibliography}{1}

\bibitem{Basdevant2012}
A.-L. Basdevant\ and\ A. Singh, Continuous-time vertex reinforced jump processes on Galton-Watson trees, Ann. Appl. Probab. {\bf 22} (2012), no.~4, 1728--1743. 

\bibitem{Benaim96}
M. Benaim, A dynamical system approach to stochastic approximations, SIAM J. Control Optim. {\bf 34} (1996), no.~2, 437--472.

\bibitem{Benaim97}
M. Bena\"{\i}m, Vertex-reinforced random walks and a conjecture of Pemantle, Ann. Probab. {\bf 25} (1997), no.~1, 361--392.

\bibitem{Benaim99}
M. Bena\"{\i}m, Dynamics of stochastic approximation algorithms, in {\it S\'{e}minaire de Probabilit\'{e}s, XXXIII}, 1--68, Lecture Notes in Math., 1709, Springer, Berlin.

\bibitem{Benaim13}
M. Benaim, O. Raimond\ and\ B. Schapira, Strongly vertex-reinforced-random-walk on a complete graph, ALEA Lat. Am. J. Probab. Math. Stat. {\bf 10} (2013), no.~2, 767--782.

\bibitem{Brandiere96}
O. Brandi\`ere\ and\ M. Duflo, Les algorithmes stochastiques contournent-ils les pi\`eges?, Ann. Inst. H. Poincar\'{e} Probab. Statist. {\bf 32} (1996), no.~3, 395--427.


\bibitem{Collevecchio2009}
A. Collevecchio, Limit theorems for vertex-reinforced jump processes on regular trees, Electron. J. Probab. {\bf 14} (2009), no. 66, 1936--1962.

\bibitem{Cotar2015}
C. Cotar\ and\ D. Thacker, Edge- and vertex-reinforced random walks with super-linear reinforcement on infinite graphs, Ann. Probab. {\bf 45} (2017), no.~4, 2655--2706.

\bibitem{Coppersmith86}
D. Coppersmith\ and\ P. Diaconis, Random walks with reinforcement, Unpublished manuscript (1986).

\bibitem{Davis02}
B. Davis\ and\ S. Volkov, Continuous time vertex-reinforced jump processes, Probab. Theory Related Fields {\bf 123} (2002), no.~2, 281--300.

\bibitem{Davis04}
B. Davis\ and\ S. Volkov, Vertex-reinforced jump processes on trees and finite graphs, Probab. Theory Related Fields {\bf 128} (2004), no.~1, 42--62.

\bibitem{Disertori14} M. Disertori, F. Merkl\ and\ S. W. W. Rolles, Localization for a nonlinear sigma model in a strip related to vertex reinforced jump processes, Comm. Math. Phys. {\bf 332} (2014), no.~2, 783--825.

\bibitem{Duflo1996} M. Duflo, {\it Algorithmes stochastiques}, Math\'{e}matiques \& Applications (Berlin), 23, Springer-Verlag, Berlin, 1996. 

\bibitem{Jacod2003}
J. Jacod\ and\ A. N. Shiryaev, {\it Limit theorems for stochastic processes}, second edition, Grundlehren der Mathematischen Wissenschaften, 288, Springer-Verlag, Berlin, 2003. 

\bibitem{Lupu2018}
T. Lupu, C. Sabot\ and\ P. Tarr\`es, Fine mesh limit of the VRJP in dimension one and Bass-Burdzy flow, Probab. Theory Related Fields {\bf 177} (2020), no.~1-2, 55--90.


\bibitem{Merkl16}
F. Merkl, S. W. W. Rolles\ and\ P. Tarr\`es, Convergence of vertex-reinforced jump processes to an extension of the supersymmetric hyperbolic nonlinear sigma model, Probab. Theory Related Fields {\bf 173} (2019), no.~3-4, 1349--1387.

\bibitem{Sabot15}
C. Sabot\ and\ P. Tarr\`es, Edge-reinforced random walk, vertex-reinforced jump process and the supersymmetric hyperbolic sigma model, J. Eur. Math. Soc. (JEMS) {\bf 17} (2015), no.~9, 2353--2378.

\bibitem{Sabot15b}
C. Sabot, P. Tarr\`es\ and\ X. Zeng, The vertex reinforced jump process and a random Schr\"{o}dinger operator on finite graphs, Ann. Probab. {\bf 45} (2017), no.~6A, 3967--3986.

\bibitem{Sabot2019}
C. Sabot\ and\ X. Zeng, A random Schr\"{o}dinger operator associated with the vertex reinforced jump process on infinite graphs, J. Amer. Math. Soc. {\bf 32} (2019), no.~2, 311--349.

\bibitem{Sabot17}
C. Sabot\ and\ X. Zeng, Hitting times of interacting drifted Brownian motions and the vertex reinforced jump process, Ann. Probab. {\bf 48} (2020), no.~3, 1057--1085. 

\bibitem{Pemantle88}
R. Pemantle, Phase transition in reinforced random walk and RWRE on trees, Ann. Probab. {\bf 16} (1988), no.~3, 1229--1241.

\bibitem{Pemantle92}
R. Pemantle, Vertex-reinforced random walk, Probab. Theory Related Fields {\bf 92} (1992), no.~1, 117--136.

\bibitem{Protter}
P. E. Protter, {\it Stochastic integration and differential equations}, second edition. Version 2.1, Stochastic Modelling and Applied Probability, 21, Springer-Verlag, Berlin, 2005.


\bibitem{Tarres04}
P. Tarr\`es, Vertex-reinforced random walk on $\Bbb Z$ eventually gets stuck on five points, Ann. Probab. {\bf 32} (2004), no.~3B, 2650--2701.


\bibitem{Volkov01}
S. Volkov, Vertex-reinforced random walk on arbitrary graphs, Ann. Probab. {\bf 29} (2001), no.~1, 66--91. 

\bibitem{Volkov06}
S. Volkov, Phase transition in vertex-reinforced random walks on $\Bbb Z$ with non-linear reinforcement, J. Theoret. Probab. {\bf 19} (2006), no.~3, 691--700.

\bibitem{Zeng16}
X. Zeng, How vertex reinforced jump process arises naturally, Ann. Inst. Henri Poincar\'{e} Probab. Stat. {\bf 52} (2016), no.~3, 1061--1075. 

\bibitem{Lai1983} T. L. Lai\ and\ C. Z. Wei, A note on martingale difference sequences satisfying the local Marcinkiewicz-Zygmund condition, Bull. Inst. Math. Acad. Sinica {\bf 11} (1983), no.~1, 1--13.
\end{thebibliography}
\end{document}